%% file: main.tex
\documentclass{article}
\usepackage[english]{babel}
\usepackage{amsmath,amssymb,amsthm,enumerate,latexsym}
\usepackage[a4paper, total={6in, 9in}]{geometry}
\usepackage[monochrome]{xcolor} 
\usepackage[hidelinks]{hyperref}

\newcommand{\assign}{:=}
\newcommand{\backassign}{=:}
\newcommand{\cdummy}{\cdot}
\newcommand{\mathd}{\mathrm{d}}

\newcommand{\tmdummy}{$\mbox{}$}
\newcommand{\tmem}[1]{{\em #1\/}}
\newcommand{\tmmathbf}[1]{\ensuremath{\boldsymbol{#1}}}
\newcommand{\tmop}[1]{\ensuremath{\operatorname{#1}}}

\newcommand{\R}{\mathbb{R}}
\newcommand{\N}{\mathbb{N}}
\newcommand{\E}{\mathbb{E}}
\renewcommand{\P}{\mathbb{P}}

\newenvironment{enumeratenumeric}{\begin{enumerate}[1.] }{\end{enumerate}}
\newenvironment{enumerateroman}{\begin{enumerate}[i.] }{\end{enumerate}}

\theoremstyle{plain}
\newtheorem{theorem}{Theorem}
\numberwithin{theorem}{section}
\newtheorem{lemma}[theorem]{Lemma}
\newtheorem{proposition}[theorem]{Proposition}
\newtheorem{corollary}[theorem]{Corollary}

\theoremstyle{definition}
\newtheorem{definition}[theorem]{Definition} 
\newtheorem{assumption}{Assumption}

\newtheorem{remark}[theorem]{Remark}
\newtheorem{example}[theorem]{Example}

\title{Energy solutions of singular SPDEs on Hilbert spaces with applications to domains with boundary conditions}

\author{Lukas Gr\"afner\thanks{Depatment of Statistics, University of Warwick, Coventry, CV4 7AL (lukas.grafner@warwick.ac.uk)},
Nicolas Perkowski\thanks{Institut f\"ur Mathematik, Freie Universit\"at Berlin, Arnimallee 7, 14195 Berlin, Germany and\\
Max-Planck-Institute for Mathematics in the Sciences, Leipzig (perkowski@math.fu-berlin.de)},
Shyam Popat\thanks{Centre de Math\'ematiques Appliqu\'ees (CMAP), \'Ecole Polytechnique, 91120 Palaiseau, France (shyam.popat@polytechnique.edu)}}

\date{\today}

\begin{document}

\maketitle

\begin{center}
    \emph{Dedicated to the memory of Giuseppe Da Prato.}
\end{center}

\input{abstract}

\input{Chapters/1_Introduction}
\input{Chapters/2_Preliminaries}
\input{Chapters/3_Existence_Proofs}
\input{Chapters/4_Uniqueness_Proof}
\input{Chapters/5_Examples}
\input{Chapters/6_Auxiliary_Computations}
\input{Chapters/7_Acknowledgments}

\bibliographystyle{amsalpha}
\bibliography{all}

\end{document}

%% file: abstract.tex
\begin{abstract}
In this paper we extend the theory of energy solutions for singular SPDEs, focusing on equations driven by highly irregular noise with bilinear nonlinearities, including scaling critical examples. By introducing Gelfand triples and leveraging infinite-dimensional analysis in Hilbert spaces together with an integration by parts formula under the invariant measure, we largely eliminate the need for Fourier series and chaos expansions.  This approach broadens the applicability of energy solutions to a wider class of SPDEs, offering a unified treatment of various domains and boundary conditions. Our examples are motivated by recent work on scaling limits of interacting particle systems.

\vspace{10pt}
\noindent \textbf{MSC2020 subject classification:} \hspace{5pt} 60H17, 60H07, 60H50\\
\noindent \textbf{Keywords:}\hspace{5pt} Energy solutions, boundary conditions, Gelfand triple, stochastic Burgers equation
\end{abstract}

%% file: Chapters/1_Introduction.tex
\section{Introduction}

The semigroup approach to SPDEs in Hilbert spaces, as developed in the seminal monograph by Da Prato and Zabczyk~\cite{DaPrato2014} is based on infinite-dimensional analysis in Hilbert spaces combined with stochastic analysis and martingale techniques. It offers a unified  approach to a wide range of equations driven by infinite-dimensional Wiener processes. The semigroup approach is notable for its versatility, addressing not only well-posedness but also more advanced aspects including  martingale problems and connections to infinite-dimensional Kolmogorov equations~\cite{DaPrato2002, DaPrato2004}, study of invariant measures and ergodicity~\cite{DaPrato1996}, stability and dynamic properties, control theory, numerical analysis, regularization by noise, reflected equations, or scaling limits. This generality has made the semigroup approach an invaluable tool for researchers tackling both theoretical and applied problems in infinite-dimensional stochastic systems.

However, for certain classes of SPDEs, particularly those involving highly singular noise and nonlinearities, the traditional semigroup approach breaks down. When the noise becomes too irregular, the solution may no longer take values in a space on which the nonlinearity is well-defined. In this case, renormalization becomes essential, and a new conceptual framework is required to handle the resulting singularities. This gap has been filled by the development of singular SPDE theory, as pioneered by Hairer's theory of regularity structures~\cite{Hairer2014} and Gubinelli-Imkeller-Perkowski's paracontrolled calculus~\cite{Gubinelli2015Paracontrolled}, both originating in Lyons's rough path theory~\cite{Lyons1998}. These pathwise techniques introduced the idea of scaling subcriticality, where the nonlinearities become small at small scales. This enables a perturbative expansion of the solution around the linearized equation that gains regularity, in the spirit of previous works by Da Prato and Debussche~\cite{DaPrato2002NavierStokes, DaPrato2003}. Singular SPDEs grew into a vibrant new area of research in SPDEs. But despite the impressive progress in the last decade, the probabilistic structure of singular SPDEs is still worse understood than for SPDEs on Hilbert spaces in the semigroup approach. In particular, the pathwise nature of these methods complicates the study of martingale properties or the relation to infinite-dimensional Kolmogorov equations, and the connection to the broader probabilistic methods used in the semigroup approach is largely absent in the existing theories of singular SPDEs.

This is where energy solutions  become relevant. Energy solutions offer a probabilistic alternative for certain singular SPDEs with a tractable probabilistic structure. They provide a martingale-based theory that is compatible with It\^o calculus. Energy solutions were first defined and constructed by Gon\c{c}alves and Jara~\cite{Goncalves2014} for the stochastic Burgers equation in the context of scaling limits for particle systems. Gubinelli and Jara~\cite{Gubinelli2013} found a slightly stronger formulation and extended the construction to stochastic Navier-Stokes equations and fractional Burgers equations. In both of these works, the question of uniqueness remained open. For the stochastic Burgers equation, Gubinelli and Perkowski~\cite{Gubinelli2018Energy} later proved well-posedness with the help of the Cole-Hopf transform that linearizes the  equation. In a subsequent work~\cite{Gubinelli2020} they extended the well-posedness theory to other equations including the fractional stochastic Burgers equation in the entire subcritical regime, by leveraging a martingale problem approach. Notably, they constructed the infinitesimal generator and solved the infinite-dimensional Kolmogorov backward equation by combining Gaussian analysis with Fourier series computations. More recently, the ongoing Ph.D. thesis of Gr\"afner~\cite{Graefner2024SPDE,Graefner2025} has advanced the theory of energy solutions by adopting a more systematic functional analytic approach that largely removes the need for Fourier series computations. This approach relies on the chaos expansion under a Gaussian or non-Gaussian reference measure, and it is able to handle  scaling critical equations like the critical fractional stochastic Burgers equation and the stochastic surface quasi-geostrophic equation. It also allows for non-stationary perturbations of drift and diffusion, time-dependent coefficients, and it provides a PDE approach (compared to the semigroup approach used here) for the Kolmogorov backward equation, as well as Markov selection results for some scaling super-critical equations.

Our current work strengthens the theory of energy solutions, broadening its scope and enhancing its robustness. We build upon the foundational ideas of SPDEs on Hilbert spaces~\cite{DaPrato2014} and combine them with the Gelfand triples from the variational approach to SPDEs~\cite{Liu2015} to rigorously define the operators involved in the stochastic dynamics, also inspired by the construction of singular operators like the Anderson Hamiltonian in the context of singular SPDEs~\cite{Allez2015}. We consider densely embedded Hilbert spaces $V\subset H \subset V^\ast$ and our focus is on equations of the form
\begin{equation}
  \partial_t u = - A u + {:} B (u, u) {:} + \sqrt{2 A} \xi . \label{eq:SPDE-intro}
\end{equation}
where \(A:V\to V^\ast\) is a symmetric dissipative operator, \(B: V\otimes_s H\to V^\ast \) (the tensor product is defined later)  is a bilinear operator with certain symmetry and antisymmetry properties, ${:}B{:}$ represents renormalization, and \(\xi\) is a space-time white noise on $H$. The candidate invariant measure for this equation is the law of the white noise on $H$, which in infinite dimensions is supported only on a  larger Hilbert space than $H$~\cite{DaPrato2014}. Therefore, some care is needed to interpret the (renormalized) nonlinearity ${:}B{:}$. We show that $B$ can be interpreted as a distribution on $L^2(\mu)$ and we use regularization by noise techniques, in particular infinite-dimensional semigroups/Kolmogorov equations, to prove weak well-posedness of energy solutions to~\eqref{eq:SPDE-intro}. This is in the spirit of works by Da Prato et al on infinite-dimensional regularization by noise~\cite{DaPrato2002Singular, DaPrato2013, DaPrato2015}. A crucial ingredient is the energy estimate of energy solutions, see Definition~\ref{def:energy}, which serves as a selection principle guaranteeing the uniqueness (see also the discussion on the selection principle in~\cite{Graefner2024}).

One of our main contributions is the introduction of Gelfand triples in this context, eliminating the need for Fourier techniques\footnote{Although Fourier methods can still be useful in concrete examples.} while extending the theory to accommodate a wide range of domains, boundary conditions and operator types. Our formulation is motivated by potential applications to the scaling limits of particle systems, inspired in part by recent work on SPDEs with regional fractional Laplacians as scaling limits by Cardoso and Gon\c{c}alves~\cite{Cardoso2024}. On the technical side, we mostly avoid the use of chaos expansions that are prominent in~\cite{Graefner2024SPDE}, and instead rely on an integration by parts rule for the reference measure. This choice not only enhances robustness but also simplifies the necessary estimates.

\subsection*{Statement of the main results}
We consider a Gelfand triple consisting of densely embedded, separable Hilbert
spaces $V \subset H \subset V^{\ast}$ and a dense subspace $S \subset V$ that
plays the role of smooth functions.
For background and further discussion of Gelfand triples, see \cite[Ch.~4.1]{Liu2015}.
We assume that $\lVert \cdummy \rVert_V \geqslant
\lVert \cdummy \rVert_H$ and we define the (semi-)norm
\[ \| h \|_{\dot{V}}^2 \assign \| h \|_V^2 - \| h \|_H^2 \ni [0, \| h
   \|_V^2], \qquad h \in V, \]
which represents a ``homogeneous norm''. The dual norm is canonically defined as
\[ \| h \|_{\dot{V}^{\ast}} \assign \sup \{ \langle h, g \rangle_{V^{\ast}, V}
   : g \in V, \| g \|_{\dot{V}} = 1 \}, \qquad h \in V^{\ast} . \]
Although it is possible to complete $V$ and $V^{\ast}$ under the (semi-)norms
$\lVert \cdot \rVert_{\dot{V}}$ and $\lVert \cdot
\rVert_{\dot{V}^{\ast}}$, respectively, to obtain Hilbert spaces $\dot{V}
\supset V$ and $\dot{V}^{\ast} \subset V^{\ast}$, this will not be necessary
for our purposes and we will rely only on the definitions of the norms.

\begin{assumption}\label{ass:A}
    Let $A : V \rightarrow V^{\ast}$ be a bounded linear operator satisfying for
all $f, g \in S$
\begin{equation} \langle A f, g \rangle_{V^{\ast}, V} = \langle f, A g \rangle_{V, V^{\ast}}
   . \end{equation}
Additionally, we assume that $A f \in H$ for all $f \in S$, meaning that
smooth functions are mapped to $H$, and we assume that there exist $c, C > 0$
such that for all $f, g \in S$
\begin{equation} | \langle A f, g \rangle_{V^{\ast}, V} | \leqslant C \| f \|_{\dot{V}} \| g
   \|_{\dot{V}}, \end{equation}
and
\begin{equation} \langle A f, f \rangle_{V^{\ast}, V} \geqslant c \| f \|_{\dot{V}}^2 . \end{equation}
By approximation, these properties extend to $f, g \in V$.
\end{assumption}

\begin{assumption}\label{ass:B}
    Let $B : V
\times V \rightarrow V^{\ast}$ be a continuous symmetric bilinear operator
with
\begin{equation} | \langle B (f, g), h \rangle_{V^{\ast}, V} | \lesssim (\| f \|_V \| g \|_H
   + \| f \|_H \| g \|_V) \| h \|_{\dot{V}}, \qquad f, g, h \in S, \end{equation}
and such that
\begin{equation} \langle B (f, f), f \rangle_{V^{\ast}, V} = 0, \qquad f \in S. \end{equation}
Again, these properties extend to $f, g, h \in V$ by approximation. We also
assume that $B$ has a continuous extension to $V \otimes_s H$, see
Definition~\ref{def:V-tensor-H}, satisfying
\begin{equation}
    \| B (\varphi)\|_{\dot{V}^{\ast}} \lesssim \| \varphi \|_{V \otimes_s H}.
\end{equation}
Moreover, we assume that there exist $\alpha>0$ and sequences $(\rho_N)_{N \in \mathbb N} \subset L(H,V)$ and $\varepsilon_N \to 0$, such that for all $f \in  S$ and $\varphi \in  S \otimes_s  S$,
  \begin{align}
    \sup_N \left\{\| \rho_N \|_{L (H, H)} + \| \rho_N \|_{L (\dot{V}, \dot{V})} + \varepsilon_N^\alpha \|\rho_N\|_{L(H,V)}\right\} & < \infty, \\
    \| (\rho_N - \text{Id})f \|_V + \| \langle B \circ \rho_N^{\otimes 2}, f \rangle_{V^{\ast}, V} -
    \langle B, f \rangle_{V^{\ast}, V} \|_{L (V \otimes_s H, \mathbb{R})} &
    \leqslant  C_1 (f) \varepsilon_N,\\
    {\color{red}\| \langle B (\varphi), (\rho_N-\text{Id})\cdummy \rangle_{V^{\ast}, V}  \|_{L (\dot{V}, \mathbb{R})}+ } \| \langle B (\rho_N^{\otimes 2}\varphi), \cdummy \rangle_{V^{\ast}, V} - \langle B
    (\varphi), \cdummy \rangle_{V^{\ast}, V} \|_{L (\dot{V}, \mathbb{R})} &
    \leqslant  C_2 (\varphi) \varepsilon_N,
  \end{align}
  where $C_1 :  S \rightarrow [0, \infty)$ and $C_2 :  S \otimes_s  S \rightarrow
  [0, \infty)$ are such that $\mathbb{E} [C_1 (D F)^2 + C_2 (D^2 F)^2] <
  \infty$ for all cylinder functions $F \in \mathcal{C}$ (to be defined below, and the expectation is with respect to the white noise on $H$).
\end{assumption}

\begin{remark}
    The assumptions on \( A \) and \( B \) are natural, while the assumption on \( (\rho_N) \) is more technical. It may be possible to avoid this assumption by finding different arguments. However, in concrete examples, there often exists a natural family of approximations \( (\rho_N) \) for which the estimates in Assumption~\ref{ass:B} hold. For instance, in the derivation of energy solutions from particle systems, \( \rho_N \) is frequently defined by integrating against a rescaled indicator function over a small set, see~\cite{Goncalves2014, Goncalves2020, Cardoso2024}.
\end{remark}

Our aim is to prove uniqueness and, under additional assumptions, existence of energy solutions to the equation
\begin{equation}
  \partial_t u = - A u + {:}B (u, u){:} + \sqrt{2 A} \xi, \label{eqn}
\end{equation}
where ${:}B(u, u){:}$ is a suitable renormalization of
the bilinearity, and $\xi$ is a space-time white noise over $H$, i.e. a
centered Gaussian process indexed by $L^2 (\mathbb{R}_+ ; H)$ with covariance
\[ \mathbb{E} [\xi (g) \xi (h)] = \int_0^{\infty} \langle g (t), h (t) \rangle
   \mathd t. \]
We analyze the equation under the candidate invariant measure $\mu$, the
law of the white noise on $H$. Under $\mu$ the process $(\eta (h))_{h \in H}$
is centered Gaussian with covariance
\[ \int \eta (h) \eta (g) \mu (\mathd \eta) = \langle h, g \rangle_H . \]
Energy solutions are weak (both analytically and probabilistically) solutions that respect the natural probabilistic structure of~\eqref{eqn}; see Definition~\ref{def:energy} for a precise formulation.

\begin{remark}\label{rmk:Hilbert-space}
The measure $\mu$ makes
sense as a Gaussian probability measure on any Hilbert space $W$ such that there is a
dense Hilbert-Schmidt embedding $\iota : H \hookrightarrow W $ since then we can consider the unique Gaussian
measure with trace-class operator $\iota \iota^{\ast}$ on $W$. Such an
embedding always exists and the specific choice of $W$ does not matter since
we deal with a $W$-valued random variable $\eta$  for which the
law of the process $(\langle \eta, \iota (h) \rangle_W)_{h \in H}$ is uniquely determined, see e.g. \cite[Sec.~4.1.2]{DaPrato2014}.
\end{remark}

\begin{theorem}[Existence of energy solutions, Theorem \ref{thm:exisence} below]\label{thm: Existence intro}
  Suppose that $S$ is a nuclear Fr{\'e}chet space and that there exists a
  complete orthonormal system $(e_n)_{n \in \mathbb{N}} \subset V$ of $H$
  consisting of eigenvectors for $A$ with eigenvalues $(\lambda_n)_{n \in
  \mathbb{N}}$.
  Let $\nu \ll \mu$ be such that $\frac{\mathd\nu}{\mathd\mu} \in L^2(\mu)$.
  Then under Assumptions \ref{ass:A} and \ref{ass:B}, there exists an energy solution $u$ to~\eqref{eqn} with $u_0 \sim \nu$.
\end{theorem}

\begin{theorem}[Uniqueness of energy solutions, Theorem \ref{thm:uniqueness} and Remark~\ref{rmk:ergodicity} below]
   Suppose Assumptions \ref{ass:A} and \ref{ass:B} are satisfied. 
   Let $\nu \ll \mu$ with $\frac{\mathd \nu}{\mathd \mu} \in L^2 (\mu)$ and let
  $(u_t)_{t \geqslant 0}$ be an energy solution to~\eqref{eqn} with $u_0 \sim \nu$. 
  Then the law of
  $u$ is uniquely determined by its initial condition, $u$ is a Markov process
  and it has $\mu$ as an invariant measure. If $\|v\|_{\dot{V}}=0$ only for $v=0 \in V$, then $u$ is ergodic.
\end{theorem}

\begin{remark}[Normalization of coefficients]
  We chose very specific coefficients for convenience, but by a small
  adaptation we can treat the equation
  \[ \partial_t u = - \nu A u + \lambda {:} B (u, u) {:} + \sqrt{\sigma^2 2 A}
     \xi, \]
  for any $\nu, \sigma^2 > 0$ and $\lambda \in \mathbb{R}$. In that case we
  simply redefine $\tilde{A} \assign \nu A$ and $\tilde{B} = \lambda B$ and
  $\tilde{\xi} = \sqrt{\frac{\sigma^2}{\nu}} \xi$ to obtain an equation with
  normalized coefficients $\partial_t u = - \tilde{A} u + : \tilde{B} (u, u) :
  + \sqrt{2 \tilde{A}} \tilde{\xi}$. The noise $\tilde{\xi}$ is no longer a
  space-time white noise over $H$, but over $\tilde{H} = H$ with rescaled norm
  $\| h \|_{\tilde{H}}^2 = \frac{\sigma^2}{\nu} \| h \|_H^2$ and $\| v
  \|_{\tilde{V}}^2 = \frac{\sigma^2}{\nu} \| v \|_V^2$. The operators
  $\tilde{A}$ and $\tilde{B}$ satisfy similar estimates with respect to
  $\tilde{V}$ as $A$ and $B$ do with respect to $V$. Therefore, the same
  existence and uniqueness results hold in that case.
\end{remark}

{\color{red}
\begin{remark}
All results of this paper, in particular the uniqueness statements, extend to the
equation obtained by adding a linear term $Cu$ to the drift in~\eqref{eqn},
where $C:\dot V\to V^\ast$ is bounded, $C(S)\subset H$, and $C$ is
antisymmetric on $S$, i.e.
\[
\langle Cf,g\rangle = -\langle Cg,f\rangle , \qquad f,g\in S
\]
(equivalently $\langle Cf,f\rangle=0$ for all $f\in S$). In the weak formulation~\eqref{eq:weak-formulation} this corresponds to the additional term
\[
\mathrm d u_t(f)=\cdots - u_t(Cf)\,\mathrm dt , \qquad f\in S.
\]
On cylinder functions $F\in\mathcal C$, the corresponding generator in \eqref{eq:generator} becomes
\[
\mathcal{L} = \mathcal{L}_0 + \mathcal{G} + \mathcal{A}_C, \qquad \mathcal{A}_C F := \delta(CDF).
\]
The operator $\mathcal A_C$ is antisymmetric on $\mathcal C$ and extends
to a bounded map $\mathcal{H}^1_\alpha\to \mathcal{H}^{-1}_\alpha$ for all $\alpha \in \mathbb R$. Hence, it is a bounded antisymmetric
perturbation of $\mathcal{L}_0$ and can be handled by a simpler variant of the
arguments used for $\mathcal{G}$.
\end{remark}
}

{\color{red}
\begin{remark}\label{rmk:discretize-A}
    In the existence result Theorem \ref{thm: Existence intro}, the condition that $A$ has an eigenbasis is technical and can likely be omitted. A possible workaround would be to approximate a self-adjoint $A$ satisfying
  Assumption~(A) by the spectral discretizations
  \[
    A^N := \sum_{i=1}^\infty \frac{i}{N}\, \mathbf 1_{\bigl(\frac{i-1}{N},\frac{i}{N}\bigr]}(A),
  \]
  where $\mathbf 1_I(A)$ denotes a spectral projection (functional calculus).
  Then $A^N$ is self-adjoint and $\sigma(A^N)\subset \frac{1}{N}\mathbb N$, hence
  $A^N$ is diagonalizable and Theorem~\ref{thm:exisence} yields an energy solution
  $u^N$ for the equation with $A$ replaced by $A^N$.
  To pass to a limit $N\to\infty$ one would need uniformity of the estimates presented below, in
  particular the tightness bounds of Lemma~\ref{lem:tightness}.
  This is in the spirit of Gubinelli--Turra~\cite{Gubinelli2019}, who approximate
  the Laplacian on $\mathbb R^2$ by Laplacians on tori of increasing size.
\end{remark}
}

\subsection*{Structure of the paper}

The paper is organized as follows. In Section \ref{sec: functional analytic setup} we provide a functional analytic setup for our problem.
Section \ref{sec: Derivation of the generator} is dedicated to deriving expressions for the operators appearing in the It\^o formula for energy solutions, $\mathcal{L}_0$ and $\mathcal{G}$, corresponding to the linear part of the dynamics and the bilinear operator $B$ as in \eqref{eqn} respectively. 
In Section \ref{sec: Function spaces and operator bounds} we derive bounds for the operators $\mathcal{L}_0$ and $\mathcal{G}$ on Sobolev type spaces based on $L^2(\mu)$.
Finally in Section \ref{sec: energy solutions} we define energy solutions in Definition \ref{def:energy} and give an alternative characterization in Theorem \ref{lem:sufficient-energy} that is more applicable in the particle systems setting. Section \ref{sec: Energy solutions and semigroup} is dedicated to the construction of energy solutions and the corresponding semigroups. In Section \ref{sec: Construction of the semigroup and resolvent} we construct the semigroup and resolvent. Energy solutions are constructed in Section \ref{sec: Construction of energy solutions}. Finally, in Section \ref{sec: Uniqueness proof: Duality of semigroup and energy solutions} we prove the uniqueness of energy solutions in Theorem \ref{thm:uniqueness}, by establishing duality between semigroup and energy solution.

The final section is devoted to examples.
For simplicity we focus on the one-dimensional Burgers equation with various
linear operators and boundary conditions. In Section~\ref{sec:examples} we
provide a detailed analysis for the (multi-component) stochastic Burgers
equation on $(0, 1)$ or $(0, \infty)$ with homogeneous Dirichlet boundary
conditions, possibly with a divergence form differential operator instead of
the usual Laplacian. We also discuss homogeneous Neumann boundary conditions,
which we can handle in a hyperviscous setting with $A = (1 - \Delta)^{\theta}$
for $\theta > 1$. In the Dirichlet case we can also choose $- A$ as the
regional fractional Laplacian of order $\gamma \in \left[ \frac{3}{2}, 2
\right)$, which is motivated by recent results on fluctuations in interacting
particle systems by \cite{Cardoso2024}.

%% file: Chapters/2_Preliminaries.tex
\section{Functional analytic setup}\label{sec: functional analytic setup}

\subsection{Derivation of the generator}\label{sec: Derivation of the generator}

In this section we derive useful expressions for the operators appearing in
the It{\^o} formula for energy solutions. We work on the probability space $\Omega=W$ with $W$ as in Remark~\ref{rmk:Hilbert-space}, where we consider the coordinate map $\eta:\Omega\to W$, $\eta(\omega) = \omega$, and $\eta(h)\assign \langle\eta,\mathcal \iota(h)\rangle_W$, $h \in H$ and $\mathcal F = \sigma(\eta(h):h \in H)$. We write $\mu$ for the measure on $(\Omega, \mathcal F)$ under which $\eta$ is a white noise on $H$, see Remark~\ref{rmk:Hilbert-space}. 

\begin{definition}[Cylinder functions]
  Any {\tmem{cylinder function}} is of the form $F (\eta) = \Phi (\eta (f))$,
  where $f \in  S^m$ and $\Phi : \mathbb{R}^m \rightarrow \mathbb{R}$ is a
  polynomial. We write $\mathcal{C}$ for the set of all cylinder functions.
\end{definition}

For a cylinder function $F (\eta) = \Phi (\eta (f))$ we define the operator
\begin{equation}
  \mathcal{L}_0 F (\eta) \assign \sum_{i = 1}^m \partial_i \Phi (\eta (f))
  \eta (- A f_i) + \frac{1}{2} \sum_{i, j = 1}^m \partial_{i j}^2
  \Phi (\eta (f)) 2 \langle A f_i, f_j \rangle_H.
  \label{eq:L0-on-cylinder}
\end{equation}
Formally, this is the differential operator that appears when we apply
It{\^o}'s formula to the (weak formulation of the) linear part of the dynamics~\eqref{eqn}
(i.e. with $B = 0$). To proceed, we need the following well known observation (see \cite{Nualart2006} for the definitions of $D$, $\delta$):

\begin{lemma}[Multiplication formula for Skorokhod integral]
  \label{lem:multiplication-formula}Let $\delta$ be the Skorokhod integral and
  let $D$ be the Malliavin derivative, both with respect to the white noise on
  $H$. Let $F \in \mathcal{C}$ and let $\varphi \in \tmop{dom} (\delta)$. Then
  \begin{equation}
    F \delta (\varphi) = \delta (F \varphi) + \langle D F, \varphi \rangle_H .
    \label{eq:generalized-multiplication}
  \end{equation}
  In particular, we have for any $h \in H$
  \begin{equation}
    F (\eta) \eta (h) = \delta (F h) (\eta) + \langle D F (\eta), h \rangle_H
    = \delta (F h) (\eta) + D_h F (\eta) . \label{eq:multiplication-formula}
  \end{equation}
\end{lemma}

\begin{proof}
  Let $G$ be another cylinder function. Then
  \begin{align*}
    \mathbb{E} [(F \delta (\varphi) - \langle D F, \varphi \rangle_H) G] & = 
    \mathbb{E} [\delta (\varphi) F G] -\mathbb{E} [\langle D F, \varphi
    \rangle_H G]\\
    & =  \mathbb{E} [\langle \varphi G, D F \rangle_H] +\mathbb{E} [\langle
    \varphi F, D G \rangle_H] -\mathbb{E} [\langle D F, \varphi \rangle_H G]\\
    & =  \mathbb{E} [\delta (\varphi F) G] .
  \end{align*}
  Since this is true for all cylinder functions $G$, the claimed identity
  follows.
\end{proof}

This gives a nice expression for $\mathcal{L}_0 F$:

\begin{corollary}
  For any cylinder function $F \in \mathcal{C}$, we have
  \begin{equation}
    \mathcal{L}_0 F = \delta (- A D F) .
  \end{equation}
\end{corollary}

\begin{proof}
  In~\eqref{eq:L0-on-cylinder} we can write
  \( \sum_{i, j = 1}^m \partial_{i j}^2 \Phi (\eta (f)) \langle A f_i, f_j
     \rangle_H = \sum_{i = 1}^m \langle D \partial_i \Phi (\eta (f)), A f_i
     \rangle_H \)
  and therefore
  \begin{align*}
    \mathcal{L}_0 F (\eta) & = \sum_{i = 1}^m \partial_i \Phi (\eta (f))
    \eta (- A f_i) + \sum_{i = 1}^m \langle D \partial_i \Phi (\eta (f)), A
    f_i \rangle_H\\
    & = \delta \left( \sum_{i = 1}^m \partial_i \Phi (\eta (f)) (- A f_i)
    \right) (\eta) = \delta (- A D F) (\eta) .\qedhere
  \end{align*}
\end{proof}

Our next goal is to derive an expression for the generator $\mathcal{G}$
corresponding to the nonlinearity $B$. Here we have to use a suitable
approximation, and even the definition of the approximation takes some work.

\begin{lemma}
  Let $\rho : H \rightarrow V$ be a bounded operator and let $v \in H$. Let
  $(e_k)_k$ be an orthonormal basis of $H$. Then the following limits exist in
  $L^2 (\mu)$ and do not depend on the specific orthonormal basis:
  \begin{equation}
    \langle {:} B^{\rho} (\eta, \eta) {:}, v \rangle_{V^{\ast}, V} \assign \lim_{N
    \rightarrow \infty} \sum_{k, \ell \leqslant N} (\eta (e_k) \eta (e_{\ell})
    - \delta_{k, \ell}) \langle B (\rho e_k, \rho e_{\ell}), \rho v
    \rangle_{V^{\ast}, V} \label{eq:B-rho}
  \end{equation}
  and, if additionally $v \in V$, the same is true for
  \begin{equation}
    \langle {:} \tilde{B}^{\rho} (\eta, \eta) {:}, v \rangle_{V^{\ast}, V} \assign
    \lim_{N \rightarrow \infty} \sum_{k, \ell \leqslant N} (\eta (e_k) \eta
    (e_{\ell}) - \delta_{k, \ell}) \langle B (\rho e_k, \rho e_{\ell}), v
    \rangle_{V^{\ast}, V}.
  \end{equation}
\end{lemma}

\begin{proof}
  The term $\sum_{k, \ell \leqslant N} (\eta (e_k) \eta (e_{\ell}) -
  \delta_{k, \ell}) \langle B (\rho e_k, \rho e_{\ell}), \rho v
  \rangle_{V^{\ast}, V}$ is in the second homogeneous chaos, and $\eta (e_k)
  \eta (e_{\ell}) - \delta_{k, \ell} = \delta (e_k \delta (e_{\ell}))$ by
  Lemma~\ref{lem:multiplication-formula}. Consider a cylinder function $G$
  that is also in the second homogeneous chaos (this is sufficient for the
  expectation below, because any other chaos would give an orthogonal
  contribution to our sum). Integrating $\delta$ by parts twice, we obtain
  with $\Pi_N h = \sum_{k \leqslant N} \langle h, e_k \rangle_{V,V^{\ast}} e_k$
  \begin{align*}
    &  \mathbb{E} \left[ \left( \sum_{k, \ell \leqslant N} (\eta (e_k) \eta
    (e_{\ell}) - \delta_{k, \ell}) \langle B (\rho e_k, \rho e_{\ell}), \rho v
    \rangle_{V^{\ast}, V} \right) G \right]\\
    &  = \sum_{k, \ell \leqslant N} \mathbb{E} [D_{e_k} D_{e_{\ell}} G
    \langle B (\rho e_k, \rho e_{\ell}), \rho v \rangle_{V^{\ast}, V}]\\
    &  =\mathbb{E} [\langle B (\rho^{\otimes 2} \Pi_N^{\otimes 2} D^2 G),
    \rho v \rangle_{V^{\ast}, V}]\\
    &  \lesssim \| \rho \|_{L (H, V)} \| \rho \|_{L (H, H)} \mathbb{E} [\|
    \Pi^{\otimes 2}_N D^2 G \|_{H \otimes_s H}^2] \| \rho \|_{L (\dot{V},
    \dot{V})} \| v \|_{\dot{V}}\\
    &  \lesssim \| \rho \|_{L (H, V)} \| \rho \|_{L (H, H)} \| \rho \|_{L
    (\dot{V}, \dot{V})} \mathbb{E} [\| D^2 G \|_{H \otimes_s H}^2] \| v
    \|_{\dot{V}}\\
    &  \simeq \| \rho \|_{L (H, V)} \| \rho \|_{L (H, H)} \| \rho \|_{L
    (\dot{V}, \dot{V})} {\| G \|_{L^2 (\mu)}}  \| v \|_{\dot{V}},
  \end{align*}
  where in the last step we used that the Malliavin derivative is a bounded
  operator when acting on a fixed chaos. This shows that $\left( \sum_{k, \ell
  \leqslant N} (\eta (e_k) \eta (e_{\ell}) - \delta_{k, \ell}) \langle B (\rho
  e_k, \rho e_{\ell}), \rho v \rangle_{V^{\ast}, V} \right)_N$ is uniformly
  bounded in $L^2 (\mu)$, and a similar argument using that $\Pi_N^{\otimes
  2}$ approximates the identity in $H \otimes_s H$ shows that it is a Cauchy
  sequence and its limit does not depend on the specific orthonormal basis
  used for the construction. The same argument works also for ${:}
  \tilde{B}^{\rho} (\eta, \eta) {:}$ in place of ${:} B^{\rho} (\eta, \eta){ :}$.
\end{proof}

\begin{remark}
    For later reference note that we can extract the following bounds from the proof:
    \begin{align}
        \|\langle {:} B^{\rho} (\eta, \eta) {:}, v \rangle_{V^{\ast}, V}\|_{L^2(\mu)} & \lesssim \| \rho \|_{L (H, V)} \| \rho \|_{L (H, H)} \| \rho \|_{L
    (\dot{V}, \dot{V})}  \| v \|_{\dot{V}},\\ \label{eq:G-rho-bound} 
        \|\langle {:} \tilde{B}^{\rho} (\eta, \eta) {:}, v \rangle_{V^{\ast}, V}\|_{L^2(\mu)} & \lesssim \| \rho \|_{L (H, V)} \| \rho \|_{L (H, H)}   \| v \|_{\dot V}.
    \end{align}
\end{remark}

\begin{lemma}\label{lem: definition of G rho}
  Let $\rho : H \rightarrow V$ be a bounded operator. We define for cylinder
  functions $F = \Phi (\eta (f)) \in \mathcal{C}$ with $\langle {:} B^{\rho}
  (\eta, \eta) {:}, f_i \rangle_{V^{\ast}, V}$ as in~\eqref{eq:B-rho}
  \begin{equation}
    \mathcal{G}^{\rho} F (\eta) \assign \sum_{i = 1}^m \partial_i \Phi (\eta
    (f)) \langle {:} B^{\rho} (\eta, \eta) {:}, f_i \rangle_{V^{\ast}, V} .
  \end{equation}
  Then we can decompose
  \begin{equation} \label{eq:G-Gplus-Gminus}
    \mathcal{G}^{\rho} F =\mathcal{G}^{\rho}_+ +\mathcal{G}^{\rho}_-,
  \end{equation}
  where with any orthonormal basis $(e_k)_{k \in \mathbb{N}}$ of $H$,
  \begin{align}
    \mathcal{G}_+^{\rho} F & =  \sum_{k, \ell \in \mathbb{N}} \delta
    (e_{\ell} \delta (e_k \langle B (\rho e_k, \rho e_{\ell}), \rho D F
    \rangle_{V^{\ast}, V})), \\
    \mathcal{G}_-^{\rho} F & =  2 \sum_{k, \ell \in \mathbb{N}} \delta (e_k
    \langle B (\rho e_k, \rho e_{\ell}), \rho D D_{e_{\ell}} F
    \rangle_{V^{\ast}, V}) . 
  \end{align}
\end{lemma}

\begin{proof}
  By definition,
  \begin{equation}
    \mathcal{G}^{\rho} F (\eta) = \sum_{k, \ell \in \mathbb{N}} \sum_{i = 1}^m
    \partial_i \Phi (\eta (f)) \langle B (\rho e_k, \rho e_{\ell}), \rho f_i
    \rangle_{V^{\ast}, V}  (\eta (e_k) \eta (e_{\ell}) - \delta_{k, \ell})
  \end{equation}
  By the multiplication rule for Skorokhod integrals in
  Lemma~\ref{lem:multiplication-formula} $\eta (e_k) \eta (e_{\ell}) -
  \delta_{k, \ell} = \delta (\delta (e_k) e_{\ell}) (\eta)$. Using once more
  multiplication rule in~\eqref{eq:generalized-multiplication} in the form $G
  \delta (\varphi) = \delta (G \varphi) + \langle D G, \varphi \rangle_H =
  \delta (G \varphi) + D_{\varphi} G$, we can thus rewrite $\mathcal{G}^{\rho}
  F$ as
  \begin{align}
    \mathcal{G}^{\rho} F & = \sum_{k, \ell \in \mathbb{N}} \delta (\langle B
    (\rho e_k, \rho e_{\ell}), \rho D F \rangle_{V^{\ast}, V} \delta (e_k)
    e_{\ell}) + \sum_{k, \ell \in \mathbb{N}} \langle D (\langle B (\rho e_k,
    \rho e_{\ell}), \rho D F \rangle_{V^{\ast}, V}), e_{\ell} \delta (e_k)
    \rangle_H \nonumber\\
    & =  \sum_{k, \ell} \delta (e_{\ell} \langle B (\rho e_k, \rho
    e_{\ell}), \rho D F \rangle_{V^{\ast}, V} \delta (e_k)) + \sum_{k, \ell}
    \langle B (\rho e_k, \rho e_{\ell}), \rho D D_{e_{\ell}} F
    \rangle_{V^{\ast}, V} \delta (e_k) \nonumber\\
    & =  \sum_{k, \ell} \delta (e_{\ell} \delta (e_k \langle B (\rho e_k,
    \rho e_{\ell}), \rho D F \rangle_{V^{\ast}, V})) + \sum_{k, \ell} \delta
    (e_{\ell} \langle B (\rho e_k, \rho e_{\ell}), \rho D D_{e_k} F
    \rangle_{V^{\ast}, V})\nonumber\\
      & + \sum_{k, \ell} \delta (e_k \langle B (\rho e_k, \rho e_{\ell}),
    \rho D D_{e_{\ell}} F \rangle_{V^{\ast}, V}) + \sum_{k, \ell} \langle B
    (\rho e_k, \rho e_{\ell}), \rho D D_{e_k} D_{e_{\ell}} F
    \rangle_{V^{\ast}, V} . 
  \end{align}
  Expanding the third order Malliavin derivative in the last term on the right hand side, we obtain
  \begin{align*}
    \sum_{k, \ell} \langle B (\rho e_k, \rho e_{\ell}), \rho D D_{e_k}
    D_{e_{\ell}} F \rangle_{V^{\ast}, V} & = \sum_{i_1, i_2, i_3 = 1}^m
    \partial_{i_1 i_2 i_3}^3 \Phi (\eta (f)) \langle B (\rho f_{i_2}, \rho
    f_{i_3}), \rho f_{i_1} \rangle_{V^{\ast}, V} .
  \end{align*}
  To proceed, we need Lemma~\ref{lem:B-circular} from the appendix, which
  shows that
  \[ \langle B (\rho f_{i_2}, \rho f_{i_3}), \rho f_{i_1} \rangle_{V^{\ast},
     V} + \langle B (\rho f_{i_1}, \rho f_{i_2}), \rho f_{i_3}
     \rangle_{V^{\ast}, V} + \langle B (\rho f_{i_3}, \rho f_{i_1}), \rho
     f_{i_2} \rangle_{V^{\ast}, V} = 0, \]
  and consequently, using the symmetry of $\partial_{i_1 i_2 i_3}^3 \Phi (\eta
  (f))$ in $i_1, i_2, i_3$,
  \begin{equation*}
    \sum_{i_1, i_2, i_3 = 1}^m \partial_{i_1 i_2 i_3}^3 \Phi (\eta (f))
    \langle B (\rho f_{i_2}, \rho f_{i_3}), \rho f_{i_1} \rangle_{V^{\ast}, V}
    = 0.
  \end{equation*}
  We remain with
  \begin{align*}
    \mathcal{G}^{\rho} F & =  \sum_{k, \ell} \delta (e_{\ell} \delta (\langle
    B (\rho e_k, \rho e_{\ell}), \rho D F \rangle_{V^{\ast}, V} e_k)) + 2
    \sum_{k, \ell} \delta (e_k (\langle B (\rho e_k, \rho e_{\ell}), \rho D
    D_{e_{\ell}} F \rangle_{V^{\ast}, V}))\\
    & =  \mathcal{G}^{\rho}_+ F +\mathcal{G}^{\rho}_- F,
  \end{align*}
  where  we used the symmetry of $B$ in the first line.
\end{proof}

The operators $\mathcal{G}_+^{\rho}$ and $\mathcal{G}_-^{\rho}$ are
antisymmetric with respect to each other:

\begin{lemma}
  \label{Gantisymmetric}If $F$ and $G$ are cylinder functions, then
  \begin{equation}
    \mathbb{E} [(\mathcal{G}_+^{\rho} F) G] = -\mathbb{E}
    [F\mathcal{G}^{\rho}_- G] .
  \end{equation}
\end{lemma}

\begin{proof}
  We use several times that $\delta$ and $D$ are adjoint operators, and in the
  fifth line we use Lemma~\ref{lem:B-circular} from the appendix together with
  the symmetry of $D_{e_k} D_{e_{\ell}} G$ in $k, \ell$
  \begin{align}
  \mathbb{E} [(\mathcal{G}_+^{\rho} F) G] & = \sum_{k, \ell} \mathbb{E}
  [\delta (e_{\ell} \delta (e_k \langle B (\rho e_k, \rho e_{\ell}), \rho D
  F \rangle_{V^{\ast}, V})) G] \nonumber \\
  & = \sum_{k, \ell} \mathbb{E} [\delta (e_k \langle B (\rho e_k, \rho
  e_{\ell}), \rho D F \rangle_{V^{\ast}, V}) D_{e_{\ell}} G] \nonumber \\
  & = \sum_{k, \ell} \mathbb{E} [\langle B (\rho e_k, \rho e_{\ell}),
  \rho D F \rangle_{V^{\ast}, V} D_{e_k} D_{e_{\ell}} G] \label{eq:Gantisymmetric-pr1} \\
  & = \sum_{k, \ell, m} \mathbb{E} [\langle B (\rho e_k, \rho e_{\ell}),
  \rho e_m \rangle_{V^{\ast}, V} D_{e_m} F D_{e_k} D_{e_{\ell}} G] \nonumber \\
  & = - 2 \sum_{k, \ell, m} \mathbb{E} [\langle B (\rho e_k, \rho e_m),
  \rho e_{\ell} \rangle_{V^{\ast}, V} D_{e_m} F D_{e_k} D_{e_{\ell}} G] \nonumber \\
  & = - 2 \sum_{k, m} \mathbb{E} [\langle B (\rho e_k, \rho e_m), \rho D
  D_{e_k} G \rangle_{V^{\ast}, V} D_{e_m} F] \nonumber \\
  & = - 2 \sum_{k, m \leqslant N} \mathbb{E} [\delta (e_m \langle B (\rho
  e_k, \rho e_m), \rho D D_{e_k} G \rangle_{V^{\ast}, V}) F] \nonumber \\
  & = -\mathbb{E} [F\mathcal{G}_- G] . \nonumber \qedhere
\end{align}
\end{proof}

\begin{remark}
    For later reference, we note that equation~\eqref{eq:Gantisymmetric-pr1} can be written as
    \begin{equation}\label{eq:G-duality-representation}
        \mathbb{E} [(\mathcal{G}_+^{\rho} F) G] = \mathbb{E} [\langle B (\rho^{\otimes 2} D^2G),
  \rho D F \rangle_{V^{\ast}, V}].
    \end{equation}
\end{remark}

To derive estimates for $\mathcal{G^\rho}$ and $\mathcal{L}_0$, we need to
introduce suitable function spaces.

\subsection{Function spaces and operator bounds}\label{sec: Function spaces and operator bounds}

Here we define Sobolev type spaces over $L^2 (\mu)$, which can be interpreted as
tensorizations of $V, H, V^{\ast}$ and also of the homogeneous spaces
$\dot{V}, \dot{V}^{\ast}$. We start with the definition of an auxiliary
operator.

\begin{definition}[Number operator/Ornstein-Uhlenbeck operator]
  The {\tmem{number operator}}, also called {\tmem{Ornstein-Uhlenbeck
  operator}}, is defined as
  \begin{equation}
    \mathcal{N} \assign \delta D, \qquad \tmop{dom} (\mathcal{N}) \assign \{ F
    \in \tmop{dom} (D) : D F \in \tmop{dom} (\delta) \} .
  \end{equation}
\end{definition}

$\mathcal{N}$ is a self-adjoint, positive semi-definite linear operator, see
\cite[p.~59]{Nualart2006}, and therefore we can define the powers $(1
+\mathcal{N})^{\alpha}$ using functional calculus. We have $\mathcal{C}
\subset \tmop{dom} ((1 +\mathcal{N})^{\alpha})$ for any $\alpha \in
\mathbb{R}$, and in fact $(1 +\mathcal{N})^{\alpha} \mathcal{C} \subset
\mathcal{C}$. This can be shown using the chaos expansion under the Gaussian
measure $\mu$, since $\mathcal{N}$ acts as multiplication with $n$ on the
$n$-th homogeneous chaos, and any element of $\mathcal{C}$ has a finite chaos
expansion.

\begin{lemma}
  \label{commutatorL0N}The operators $\mathcal{L}_0$ and $\mathcal{N}$ commute
  on cylinder functions:
  \begin{equation}\label{commutatorestimate1}
    [\mathcal{L}_0, \mathcal{N}] F = \mathcal{L}_0 \mathcal{N} F - \mathcal{N}
    \mathcal{L}_0 F = 0,
  \end{equation}
  for all $F \in \mathcal{C}$.
\end{lemma}

\begin{proof}
  By Lemma \ref{Ncalculus}, we have
  \[
    \mathcal{L}_0 \mathcal{N} F = - \delta (A D \mathcal{N} F) = - \delta (A
    (\mathcal{N} + 1) D \mathcal{} F) = - \delta ((\mathcal{N} + 1) A D
    \mathcal{} F) = - \mathcal{N} \delta (A D \mathcal{} F) \qedhere.
  \]
  \end{proof}

\begin{definition}\label{def: Sobolev type spaces}
  For any cylinder function $F \in \mathcal{C}$ and any $\alpha \in
  \mathbb{R}$ we define
  \begin{align}
    \| F \|_{\dot{\mathcal{H}}^1_{\alpha}}^2 & \assign  \mathbb{E} [\| D (1
    +\mathcal{N})^{\alpha} F \|_{\dot{V}}^2], \\
    \| F \|_{\mathcal{H}^1_{\alpha}}^2 & \assign  \| (1
    +\mathcal{N})^{\alpha} F \|_{L^2 (\mu)}^2 + \| F
    \|_{\dot{\mathcal{H}}^1_{\alpha}}^2, \\
    \| F \|_{\dot{\mathcal{H}}^{- 1}_{- \alpha}}^2 & \assign \sup \{
    \mathbb{E} [F G] : G \in \mathcal{C}, \| G
    \|_{\dot{\mathcal{H}}^1_{\alpha}} = 1 \}, \\
    \| F \|_{\mathcal{H}^{- 1}_{- \alpha}}^2 & \assign \sup \{ \mathbb{E} [F
    G] : G \in \mathcal{C}, \| G \|_{\mathcal{H}^1_{\alpha}} = 1 \} . 
  \end{align}
  We also define the Hilbert spaces $\mathcal{H}^1_{\alpha}$ and
  $\mathcal{H}^{- 1}_{- \alpha}$ as the completion of $(\mathcal{C}, \|
  \cdummy \|_{\mathcal{H}^1_{\alpha}})$ and $(\mathcal{C}, \| \cdummy
  \|_{\mathcal{H}^{- 1}_{- \alpha}})$, respectively.
\end{definition}

The homogeneous spaces $\dot{\mathcal{H}}^1_0$ and $\dot{\mathcal{H}}^{-1}_0$ play an important role when studying long time fluctuations in Markov processes, see the monograph~\cite{Komorowski2012}. Since here we are interested in solvability and not in the long time behavior (beyond ergodicity), we do not need those spaces and only the norms $\lVert \cdot \rVert_{\dot{\mathcal{H}}^1_{\alpha}}^2$ and $\lVert \cdot \rVert_{\dot{\mathcal{H}}^{-1}_{\alpha}}^2$.

\begin{lemma}[Bounds for $\mathcal{L}_0$]
  \label{lem:L0-bounds} The operator $\mathcal L_0$ satisfies the following upper and lower bounds on cylinder functions:
  \begin{equation}
    \| (1 -\mathcal{L}_0) F \|_{\mathcal{H}^{- 1}_{0}} \lesssim \| F
    \|_{\mathcal{H}^1_{0}}, \qquad F \in \mathcal{C},
  \end{equation}
  and
  \begin{equation}
    \langle (1 -\mathcal{L}_0) F, F \rangle_{\mathcal{H}^{- 1}_0,
    \mathcal{H}^1_0} \gtrsim \| F \|_{\mathcal{H}^1_0}^2, \qquad F \in \mathcal C. \label{eq:L0-lower}
  \end{equation}
  Consequently, there is a unique extension of
  $\mathcal{L}_0$ to a bounded linear operator from $\mathcal{H}^1_{0}$
  to $\mathcal{H}^{- 1}_{0}$, which we denote again by $\mathcal{L}_0$,
  and which satisfies the lower bound~\eqref{eq:L0-lower} for all $F \in \mathcal H^1_0$.
\end{lemma}

\begin{proof}
  We use the
  bound $| \langle A v, w \rangle_{V^{\ast}, V} | \leqslant C \| v
  \|_{\dot{V}} \| w \|_{\dot{V}}$ together with the Cauchy-Schwarz inequality
  to obtain the claimed upper bound for $\| (1 -\mathcal{L}_0) F
  \|_{\mathcal{H}^{- 1}_0}$ by testing against $G \in \mathcal C$:
  \begin{align}
    \mathbb{E} [(1 -\mathcal{L}_0) F G] & =  \mathbb{E} [F G] +\mathbb{E}
    [\delta (A D F) G] \nonumber\\
    & =  \mathbb{E} [F G] +\mathbb{E} [\langle A D F, D G \rangle_H]
    \nonumber\\
    & \leqslant  \mathbb{E} [F^2]^{1 / 2} \mathbb{E} [G^2]^{1 / 2} +
    C\mathbb{E} [\| D F \|_{\dot{V}}^2]^{1 / 2} \mathbb{E} [\| D G
    \|_{\dot{V}}]^{1 / 2} \nonumber\\
    & \lesssim  \| F \|_{\mathcal{H}^1_0} \| G \|_{\mathcal{H}^1_0} . 
  \end{align}
  For the lower bound, we similarly estimate
  \begin{align}
    \mathbb{E} [((1 -\mathcal{L}_0) F) F] & = \mathbb{E} [F^2] +\mathbb{E}
    [\langle A D F, D F \rangle_H] \nonumber\\
    & \geqslant \mathbb{E} [F^2] + c\mathbb{E} [\| D F \|_{\dot{V}}^2]
    \nonumber\\
    & \gtrsim \| F \|_{\mathcal{H}^1_0}^2. \qedhere
  \end{align}
\end{proof}

\begin{definition}[$V \otimes_s H$]
  \label{def:V-tensor-H}For $m \in \mathbb{N}$, symmetric $a = (a_{i j})_{i, j
  = 1, \ldots, m} \in \mathbb{R}^{m \times m}$, and $v_1, \ldots, v_m \in V$
  we define
  \[ \left\| \sum_{i, j = 1}^m a_{i j} v_i \otimes v_j \right\|_{V \otimes_s
     H}^2 \assign \sum_{i_1, j_1, i_2, j_2 = 1}^m a_{i_1 j_1} a_{i_2 j_2}
     \langle v_{i_1}, v_{i_2} \rangle_V \langle v_{j_1}, v_{j_2} \rangle_H, \]
  and we write $V \otimes_s H \subset H \otimes_s H$ for the completion of
  $\left\{ \sum_{i, j = 1}^m a_{i j} v_i \otimes v_j \right\}$ with respect to
  $\| \cdummy \|_{V \otimes_s H}$.
\end{definition}

In the following, we will strongly rely on our assumption that the map
\begin{equation}
  V \otimes_s H \ni \sum_{i, j = 1}^m a_{i j} v_i \otimes v_j \mapsto B \left(
  \sum_{i, j = 1}^m a_{i j} v_i \otimes v_j \right) \assign \sum_{i, j = 1}^m
  a_{i j} B (v_i, v_j) \in \dot{V}^{\ast}
\end{equation}
has a unique continuous extension to all of $V \otimes_s H$, which we still
denote by $B$. We first derive uniform bounds for $\mathcal{G}^{\rho}$, then
we will discuss the convergence of $(\mathcal{G}^{\rho_N})_N$ if $\rho_N$ is a
sequence of regularization operators that converge to the identity.

\begin{lemma}[Uniform bounds for $\mathcal{G}^\rho$]
  \label{lem:G-bounds} Let $\rho: H \rightarrow V$ be a bounded operator.
  For $\alpha \in \{ 0, 1 \}$ the following estimates hold:
  \begin{equation}\label{eq:G-bounds}
    \| \mathcal{G}^{\rho} F \|_{\mathcal{H}^{- 1}_{\alpha - 1}} \lesssim (\|
    \rho \|_{L (H, H)} + \| \rho \|_{L (\dot{V}, \dot{V})})^3 \| F
    \|_{\mathcal{H}^1_{\alpha}}, \qquad F \in \mathcal{C},
  \end{equation}
  as well as the commutator estimate
  \begin{equation} \| [\mathcal{N}, \mathcal{G}^\rho] F \|_{\mathcal{H}^{- 1}_{\alpha - 1}}
     \lesssim (\|
    \rho \|_{L (H, H)} + \| \rho \|_{L (\dot{V}, \dot{V})})^3 \| F \|_{\mathcal{H}^1_{\alpha}}, \qquad F \in \mathcal{C}.
  \end{equation}
\end{lemma}

\begin{proof}
  To bound $\mathcal{G}_+^{\rho} F$ we consider a cylinder function $G \in
  \mathcal{C}$ and use equation~\eqref{eq:G-duality-representation}  and in the third line we apply
  Lemma~\ref{lem:tensor-operator} from the appendix and use that $\| \rho
  \|_{L (V, V)} \leqslant \| \rho \|_{L (H, H)} + \| \rho \|_{L (\dot{V},
  \dot{V})}$
  \begin{align}
    \mathbb{E} [(\mathcal{G}_+^{\rho} F) G] & =  \mathbb{E} [\langle B (\rho^{\otimes 2} D^2G),
  \rho D F \rangle_{V^{\ast}, V}] \nonumber\\
    & \lesssim  \mathbb{E} [\| \rho^{\otimes 2} D^2 G \|_{V \otimes_s
    H}^2]^{1 / 2} \mathbb{E} [\| \rho D F \|_{\dot{V}}^2]^{1 / 2} \nonumber\\
    & \lesssim  (\| \rho \|_{L (H, H)} + \| \rho \|_{L (\dot{V},
    \dot{V})})^3 (\mathbb{E} [\| D^2 G \|_{\dot{V} \otimes_s H}^2] +\mathbb{E}
    [\| D^2 G \|_{H \otimes_s H}^2])^{1 / 2} \mathbb{E} [\| D F
    \|_{\dot{V}}^2]^{1 / 2} \nonumber\\
    & \lesssim  (\| \rho \|_{L (H, H)} + \| \rho \|_{L (\dot{V},
    \dot{V})})^3 (\| \tmmathbf{1}_{\mathcal{N} \geqslant 1} (\mathcal{N}-
    1)^{1 / 2} G \|_{\dot{\mathcal{H}}^1_0}^2 + \| \mathcal{N}^{1 / 2}
    (\mathcal{N}- 1)^{1 / 2} G \|_{L^2 (\mu)}^2)^{1 / 2} \nonumber\\
    &   \qquad \cdot \| F \|_{\dot{\mathcal{H}}^1_0} \nonumber\\
    & \lesssim  (\| \rho \|_{L (H, H)} + \| \rho \|_{L (\dot{V},
    \dot{V})})^3 \| F \|_{\dot{\mathcal{H}}^1_0} \| G \|_{\mathcal{H}^1_1}, 
  \end{align}
  using Lemma~\ref{lem:second-Malliavin} which follows below this proof in the
  sixth line. As $(1 +\mathcal{N}) \delta = \delta (2 +\mathcal{N})$ and $(3
  +\mathcal{N}) D = D (2 +\mathcal{N})$, see Lemma~\ref{Ncalculus} in the
  appendix, this also yields
  \begin{align}
    \mathbb{E} [(\mathcal{G}_+^{\rho} F) G] & =  \sum_{k, \ell \in
    \mathbb{N}} \mathbb{E} [(1 +\mathcal{N}) \delta (\delta (\langle B (e_k,
    e_{\ell}), D F \rangle_{\dot{V}^{\ast}, \dot{V}} e_k) e_{\ell}) (1
    +\mathcal{N})^{- 1} G] \nonumber\\
    & =  \sum_{k, \ell \in \mathbb{N}} \mathbb{E} [\langle B (e_k,
    e_{\ell}), D (2 +\mathcal{N}) F \rangle_{\dot{V}^{\ast}, \dot{V}} D_{e_k}
    D_{e_{\ell}} (1 +\mathcal{N})^{- 1} G] \nonumber\\
    & =  \mathbb{E} [(\mathcal{G}_+ (2 +\mathcal{N}) F) (1 +\mathcal{N})^{-
    1} G] \nonumber\\
    & \lesssim  (\| \rho \|_{L (H, H)} + \| \rho \|_{L (\dot{V},
    \dot{V})})^3 \| (2 +\mathcal{N}) F \|_{\dot{\mathcal{H}}^1_0}, \| (1
    +\mathcal{N})^{- 1} G \|_{\mathcal{H}^1_1} \nonumber\\
    & \simeq  (\| \rho \|_{L (H, H)} + \| \rho \|_{L (\dot{V}, \dot{V})})^3
    \| F \|_{\dot{\mathcal{H}}^1_1} \| G \|_{\mathcal{H}^1_0} . 
  \end{align}
  We estimate $\mathcal{G}_-^{\rho}$ by duality with $\mathcal{G}^{\rho}_+$
  \[ \mathbb{E} [(\mathcal{G}^{\rho}_- F) G] =\mathbb{E}
     [F\mathcal{G}^{\rho}_+ G] \lesssim (\| \rho \|_{L (H, H)} + \| \rho \|_{L
     (\dot{V}, \dot{V})})^3 (\| F \|_{\mathcal{H}^1_1} \| G
     \|_{\dot{\mathcal{H}}^1_0}) \wedge (\| F \|_{\mathcal{H}^1_0} \| G
     \|_{\dot{\mathcal{H}}^1_1}) . \]
   The commutator estimate is the same as in~\cite{Graefner2024SPDE}: We just showed
  \[ \mathcal{N}\mathcal{G}^{\rho}_+ =\mathcal{G}_+^{\rho} (1 +\mathcal{N})
     \qquad \Rightarrow \qquad [\mathcal{N}, \mathcal{G}^{\rho}_+]
     =\mathcal{G}^{\rho}_+, \]
  and by duality this also yields $[\mathcal{N}, \mathcal{G}^{\rho}_-] =
  -\mathcal{G}^{\rho}_-$.
\end{proof}

\begin{lemma}
  \label{lem:second-Malliavin}The following identity holds for all $F \in
  \mathcal{C}$:
  \[ \mathbb{E} [\| D^2 F \|_{\dot{V} \otimes_s H}^2] = \|
     \tmmathbf{1}_{\mathcal{N} \geqslant 1} (\mathcal{N}- 1)^{1 / 2} F
     \|^2_{\dot{\mathcal{H}}^1_0} . \]
\end{lemma}

\begin{proof}
  This is the only proof where we explicitly use the chaos expansion under $\mu$, and it could be avoided with some additional work. The chaos expansion is $F =
  \sum_{n = 0}^{\infty} W_n (\varphi_n)$ with $\varphi_n \in H^{\otimes n}_s$.
  Then
  {\allowdisplaybreaks
  \begin{align*} 
    \mathbb{E} [\| D^2 F \|_{\dot{V} \otimes_s H}^2] & = \sum_{k_1, k_2,
    \ell_1, \ell_2} \mathbb{E} \left[ D_{e_{\ell_1}} D_{e_{k_1}} F
    D_{e_{\ell_2}} D_{e_{k_2}} F \right] \langle e_{\ell_1}, e_{\ell_2}
    \rangle_{\dot{V}} \langle e_{k_1}, e_{k_2} \rangle_H\\
    & = \sum_{k, \ell_1, \ell_2} \sum_{n = 2}^{\infty} n^2 (n - 1)^2
    \mathbb{E} [W_{n - 2} (\varphi_n (e_{\ell_1}, e_k, \cdummy)) W_{n - 2}
    (\varphi_n (e_{\ell_2}, e_k, \cdummy))] \langle e_{\ell_1}, e_{\ell_2}
    \rangle_{\dot{V}}\\
    & = \sum_{k, \ell_1, \ell_2} \sum_{n = 2}^{\infty} n^2 (n - 1)^2 (n -
    2) ! \langle \varphi_n (e_{\ell_1}, e_k, \cdummy), \varphi_n (e_{\ell_2},
    e_k, \cdummy) \rangle_{H^{\otimes (n - 2)}_s} \langle e_{\ell_1},
    e_{\ell_2} \rangle_{\dot{V}}\\
    & = \sum_{\ell_1, \ell_2} \sum_{n = 2}^{\infty} n^2 (n - 1) (n - 1) !
    \langle \varphi_n (e_{\ell_1}, \cdummy), \varphi_n (e_{\ell_2}, \cdummy)
    \rangle_{H^{\otimes (n - 1)}_s} \langle e_{\ell_1}, e_{\ell_2}
    \rangle_{\dot{V}}\\
    & = \sum_{n = 1}^{\infty} (n - 1) \mathbb{E} \left[ D_{e_{\ell_1}} W_n
    (\varphi_n) D_{e_{\ell_2}} W_n (\varphi_n) \right] \langle e_{\ell_1},
    e_{\ell_2} \rangle_{\dot{V}}\\
    & = \sum_{n = 1}^{\infty} \mathbb{E} \left[ D_{e_{\ell_1}}
    (\mathcal{N}- 1)^{1 / 2} W_n (\varphi_n) D_{e_{\ell_2}} (\mathcal{N}-
    1)^{1 / 2} W_n (\varphi_n) \right] \langle e_{\ell_1}, e_{\ell_2}
    \rangle_{\dot{V}}\\
    & = \| \tmmathbf{1}_{\mathcal{N} \geqslant 1} (\mathcal{N}- 1)^{1 / 2}
    F \|^2_{\dot{\mathcal{H}}^1_0} . \qedhere
  \end{align*}}
\end{proof}

Our estimates for $\mathcal{G}^{\rho}_+$ and $\mathcal{G}^{\rho}_-$ are based
on duality arguments, which means that we need strong assumptions on
$(\rho_N)$ to obtain the strong convergence of $(\mathcal{G}^{\rho_N} F)_N$ in
$\mathcal{H}^{- 1}_0$ in the next lemma. The weak convergence in
$\mathcal{H}^{- 1}_0$ would require much weaker assumptions, essentially only
$ \rho_N h \to h$ in $H$ for all $h\in H$, and $\rho_N v \to v$ in $\dot V$ for all $v \in \dot V$.

\begin{lemma}[Approximation of $\mathcal G F$]
  \label{lem:G-approximation}Let $(\rho_N)_{N \in \mathbb{N}} \subset L (H,
  V)$ be such that $\sup_N \| \rho_N \|_{L (H, H)} + \| \rho_N \|_{L (\dot{V},
  \dot{V})} < \infty$ and for all $f \in  S$
  \begin{align*}
    \| (\rho_N-\text{Id}) f  \|_V + \| \langle B\circ \rho_N^{\otimes 2}, f \rangle_{V^{\ast}, V} -
    \langle B, f \rangle_{V^{\ast}, V} \|_{L (V \otimes_s H, \mathbb{R})} &
    \leqslant  C_1 (f) \varepsilon_N,\\
    {\color{red}\| \langle B (\varphi), (\rho_N-\text{Id})\cdummy \rangle_{V^{\ast}, V}  \|_{L (\dot{V}, \mathbb{R})}+}\| \langle B (\rho_N^{\otimes 2}\varphi), \cdummy \rangle_{V^{\ast}, V} - \langle B
    (\varphi), \cdummy \rangle_{V^{\ast}, V} \|_{L (\dot{V}, \mathbb{R})} &
    \leqslant C_2 (\varphi) \varepsilon_N,
  \end{align*}
  where $C_1 :  S \rightarrow [0, \infty)$ and $C_2 :  S \otimes_s  S \rightarrow
  [0, \infty)$ are such that $\mathbb{E} [C_1 (D F)^2 + C_2 (D^2 F)^2] <
  \infty$ for all $F \in \mathcal{C}$ and where $\varepsilon_N \rightarrow 0$.
  Then for any cylinder function $F \in \mathcal{C}$ the limit
  \[ \mathcal{G}F = \lim_{N \rightarrow \infty} \mathcal{G}^{\rho_N} F \]
  exists in $\mathcal{H}^{- 1}_0$. The limit does not depend on the specific
  approximation, it satisfies the same bounds as in Lemma~\ref{lem:G-bounds}
  with the $\rho$-dependent factors on the right hand side equal to $1$, and
  it is also antisymmetric on cylinder functions. Moreover, for $f \in  S$ we also have the
  convergence
  \begin{equation}\label{eq:B-tilde-speed}
    \|\mathcal{G} (\eta (f)) - \langle {:}
     \tilde{B}^{\rho_N} (\eta, \eta) {:}, f \rangle_{V^{\ast}, V}\|_{\mathcal H^{-1}_0} \lesssim \varepsilon_N C_1(f),
  \end{equation}
  where we recall that formally $\langle {:}
  \tilde{B}^{\rho_N} (\eta, \eta) {:}, f \rangle_{V^{\ast}, V} = \langle {:} B
  (\rho_N \eta, \rho_N \eta) {:}, f \rangle_{V^{\ast}, V}$, i.e. the
  approximation $\rho_N$ is not acting on $f$.
\end{lemma}

\begin{proof}
  By similar computations as in Lemma~\ref{lem:G-bounds} we obtain for $G \in
  \mathcal{C}$
  \begin{align}
    \mathbb{E} [(\mathcal{G}_+^{\rho_N} F -\mathcal{G}^{\rho_M}_+ F) G] & = 
    \mathbb{E} [\langle B (\rho_N^{\otimes 2}D^2 (1 +\mathcal{N})^{- 1} G) -
    B (\rho_M^{\otimes 2}D^2 (1 +\mathcal{N})^{- 1} G), \rho_N D (2 +\mathcal{N}) F
    \rangle_{\dot{V}^{\ast}, \dot{V}}] \nonumber\\
    &\quad + \mathbb{E} [\langle  B (\rho_M^{\otimes 2}D^2 (1 +\mathcal{N})^{- 1} G), (\rho_N - \rho_M) D (2 +\mathcal{N}) F
    \rangle_{\dot{V}^{\ast}, \dot{V}}] \nonumber \\
    & \lesssim  \mathbb{E} [(\varepsilon_N + \varepsilon_M) \| D^2 (1
    +\mathcal{N})^{- 1} G \|_{V \otimes_s H} C_1 (D (2 +\mathcal{N}) F)]
    \nonumber\\
    & \leqslant  (\varepsilon_N + \varepsilon_M) \| G \|_{\mathcal{H}^1_0}
    \mathbb{E} [C_1 (D (2 +\mathcal{N}) F)^2]^{1 / 2} \label{eq:G-approximation-pr1}
  \end{align}
  and analogously
  \begin{align*}
    \mathbb{E} [(\mathcal{G}_-^{\rho_N} F -\mathcal{G}^{\rho_M}_- F) G] &
    \lesssim  \mathbb{E} [(\varepsilon_N + \varepsilon_M) C_2 (D^2 F) \| D G
    \|_{\dot{V}}]\\
    & \leqslant  (\varepsilon_N + \varepsilon_M) \mathbb{E} [C_2 (D^2
    F)^2]^{1 / 2} \| G \|_{\dot{\mathcal{H}}^{1}_0},
  \end{align*}
  which shows that $(\mathcal{G}^{\rho_N} F)_N$ is Cauchy in $\mathcal{H}^{-
  1}_0$ and the limit does not depend on the specific approximation, it
  satisfies the bounds of Lemma~\ref{lem:G-bounds} and it is antisymmetric on
  cylinder functions.

  The convergence of $\langle : \tilde{B}^{\rho_N} (\eta, \eta) :, f \rangle_{V^{\ast}, V} = \mathcal{G}^{\rho_N} (\eta (f)) + \langle {:} \tilde{B}^{\rho_N} (\eta, \eta) {:}, f - \rho_N f \rangle_{V^{\ast}, V}$ is similar.
\end{proof}

{\color{red}
\begin{definition}
    We write
    \begin{equation}\label{eq:generator}
        \mathcal{L} := \mathcal{L}_0 + \mathcal{G}
    \end{equation}
    for the formal generator of the solution to~\eqref{eqn}. It boundedly maps $\mathcal{H}^1_\alpha$ to $\mathcal{H}^{-1}_{\alpha-1}$.
\end{definition}
}

\subsection{Energy solutions}\label{sec: energy solutions}

\begin{definition}[Energy solution]\label{def:energy} We make Assumptions (A) and (B).
  Let $(u_t (f))_{t \geqslant 0, f \in  S}$ be an adapted  stochastic process on some
  filtered probability space $(\Omega, \mathcal{F}, (\mathcal{F}_t)_{t\ge0}, \mathbb{P})$ satisfying the usual conditions. We call $u$ an energy
  solution to the equation
  \[ \partial_t u = - A u + {:} B (u, u) {:} + \sqrt{2 A} \xi, \]
  if the following conditions are satisfied:
  \begin{enumerateroman}
    \item Path regularity: For all $f \in  S$ the process $t \mapsto u_t (f)$
    is continuous.
    
    \item \label{Incompressibility}{\color{red}Uniform $L^2$ bound for the density:} $\sup_{t \leqslant T} |
    \mathbb{E} [F (u_t)] | \lesssim \| F \|_{L^2 (\mu)}$ for all $T > 0$ and $F \in \mathcal C$.
    
    \item \label{Energy estimate}Energy estimate: For all $F \in \mathcal{C}$
    and for all $T > 0$
    \begin{equation}\label{eq:energy-estimate}
        \mathbb{E} \left[ \sup_{t \leqslant T} \left| \int_0^t F (u_s) \mathd s\right| \right] \lesssim (T^{1 / 2} + T) \| F \|_{\mathcal{H}^{- 1}_0}.
    \end{equation}
    In particular, there exists a unique continuous extension to $\mathcal{H}^{- 1}_0$ of the map
    \[ I : \mathcal{C} \rightarrow \bigcap_{T>0} L^1 (\Omega, C ([0,T],
       \mathbb{R})), \qquad I (F)_t = \int_0^t F (u_s) \mathd s. \]
    \item Weak solution: For $f \in  S$ the process
    \begin{equation}\label{eq:weak-formulation}
        M^f_t \assign u_t (f) - u_0 (f) + \int_0^t u_s(A f) \mathd s - I (\langle {:} B (\cdummy, \cdummy) {:}, f
       \rangle_{V^{\ast}, V})_t, \qquad t \geqslant 0,
    \end{equation}
    is a continuous martingale with quadratic variation $[M^f]_t = t 2 \langle
    A f, f \rangle_H$, and there exist $p<2$ and a sequence $(\rho_N)_{N \in \mathbb N}\subset L(H,V)$ such that for all $T>0$ the following convergence holds in probability:
    \[
        \lim_{N \rightarrow \infty} \left\|I (\langle {:} B (\cdummy, \cdummy) {:}, f \rangle_{V^{\ast}, V})_{\cdot} - \int_0^\cdot \langle {:} \tilde B^{\rho_N} (u_s, u_s) {:}, f
       \rangle_{V^{\ast}, V} \mathd s \right\|_{[0,T],p-\mathrm{var}} = 0,
    \]
    where
    \[
        \| I \|_{[0,T],p-\mathrm{var}} = \sup \left\{ \left(\sum_{k=0}^{n-1} |I_{t_{k+1}} - I_{t_k}|^p\right)^{1/p}: n \in \mathbb N, 0=t_0 < \dots < t_n =T \right\}.
    \]
  \end{enumerateroman}
\end{definition}

In the formulation of the weak solution property iv. we used implicitly that for any $F\in L^2(\mu)$ the integral $\int_0^t F(u_s)\mathd s$ makes sense by the path regularity and by the uniform $L^2$ bound for the density of $u_t$: For $F \in \mathcal C$ this follows from path regularity, and for $F \in L^2(\mu)$ we can find $(F^N) \subset \mathcal{C}$ that converges to $F$ in $L^2(\mu)$, so that by the uniform $L^2$ bound for the density $(\int_0^\cdot F^N(u_s) \mathd s)_N$ is Cauchy in $L^1(\Omega, C([0,T],\mathbb R))$ and we define $\int_0^\cdot F(u_s) \mathd s$ as the limit (which does not depend on the approximating sequence). We also used that $\langle {:} B (\cdummy, \cdummy) {:}, f  \rangle_{V^{\ast}, V} = \mathcal G \eta(f) \in \mathcal H^{-1}_0$, which holds because $\eta(f) \in \mathcal H^1_1$.

\begin{remark}
    The uniform $L^2$ bound for the density and the energy estimate are formulated in a way that corresponds to energy solutions that at time $0$ have a distribution with density with respect to $\mu$ that is in $L^2(\mu)$. It would be possible to relax this to an $L^1$ density by replacing the control of moments by a control of probabilities, see \cite[Def.~13]{Graefner2024}.
\end{remark}

In applications to scaling limits of particle systems the following characterization of energy solutions is sometimes useful. 

{\color{red}
\begin{theorem}[Sufficient conditions for an energy solution]\label{lem:sufficient-energy} We make Assumptions (A) and (B) and additionally assume that $S$ is a core for the operator $A$. Any adapted stochastic process $(u_t (f))_{t \geqslant 0,
  f \in  S}$ that satisfies all of the following conditions is an energy solution in the sense of Definition~\ref{def:energy}:
  \begin{enumerateroman}
    \item Path regularity: For all $f \in  S$ the process $t \mapsto u_t (f)$
    is continuous.
    
    \item Stationarity: $u_t \sim \mu$ for all $t \geqslant 0$.

    \item Weak solution: There exist $p<2$ and a sequence $(\rho_N)_{N \in \mathbb N}\subset L(H,V)$ such that for all $f \in S$ there is a process $\mathcal B^f$ with
    \begin{align*}
        \lim_{N \rightarrow \infty}\mathbb{P}\left( \left\|\mathcal B^f - \int_0^\cdot \langle {:} \tilde B^{\rho_N} (u_s, u_s) {:}, f \rangle_{V^{\ast}, V} \mathd s \right\|_{[0,T],p-\mathrm{var}}>\varepsilon\right) & = 0, \qquad \varepsilon, T>0,\\
        \lim_{N\to \infty} \left\| \langle {:} B (\eta, \eta) {:}, f \rangle_{V^{\ast}, V} - \langle {:} \tilde B^{\rho_N} (\eta, \eta) {:}, f \rangle_{V^{\ast}, V}\right\|_{\mathcal{H}^{-1}_0}& = 0.
    \end{align*}
    Moreover, the process
    \[ M^f_t \assign u_t (f) - u_0 (f) + \int_0^t u_s(Af)\mathd s - \mathcal B^f_t, \qquad t \geqslant 0, \]
    is a continuous martingale with quadratic variation $[M^f]_t =  2t \langle
    A f, f \rangle_H$.
    
    \item Behavior under time-reversal: for $T > 0$ the processes $\hat{u}_t (f)
    = u_{T - t} (f)$, and $\hat{\mathcal B}^f_t = (\mathcal B^f_T - \mathcal B^f_{T-t})$, $t \in [0, T]$, are such that
    \[
        \hat{M}^f_t \assign \hat{u}_t (f) - \hat{u}_0 (f) + \int_0^t \hat{u}_s(Af)\mathd s + \hat{\mathcal B}^f_t, \qquad t \in [0,T],
    \]
    is a continuous martingale in the filtration generated by $\hat{u}$, with quadratic variation $[\hat{M}^f]_t =  2 t \langle  A f, f \rangle_H$.
  \end{enumerateroman}
  If the sequence $(\rho_N)_{N\in \mathbb N}$ is as in Assumption~(B), then in point iii. it suffices to verify uniform convergence on compacts of $\int_0^\cdot \langle {:} \tilde B^{\rho_N} (u_s, u_s) {:}, f \rangle_{V^{\ast}, V} \mathd s$ to $\mathcal B^f$ and that $\mathcal B^f$ is of vanishing quadratic variation. It is not necessary to verify the convergence of  $\langle {:} \tilde B^{\rho_N} (\eta, \eta) {:}, f \rangle_{V^{\ast}, V}$ to $\langle {:} B (\eta, \eta) {:}, f \rangle_{V^{\ast}, V}$.
\end{theorem}

\begin{proof}
    There are two statements which are not obvious. First, that i.-iv. imply an improved version of the energy estimate, which is shown in Lemma~\ref{lem:sufficient-energy-3} below, and which establishes that $u$ is an energy solution. The lemma does not actually require the full strength of iii., and in particular the only requirement on $\mathcal B^f$ is that it is of vanishing quadratic variation.

    Lemma~\ref{lem:sufficient-energy-3} then allows us to show the second claim: Assume that the sequence $(\rho_N)_{N\in \mathbb N}$ is as in Assumption~(B). Then the convergence of $\langle {:} \tilde B^{\rho_N} (\eta, \eta) {:}, f \rangle_{V^{\ast}, V}$ to $\langle {:} B (\eta, \eta) {:}, f \rangle_{V^{\ast}, V}$ in $\mathcal{H}^{-1}_0$ follows from Lemma~\ref{lem:G-approximation}. If in point iii. the process $\mathcal B^f$ is of vanishing quadratic variation and we have uniform convergence on compacts of $\int_0^\cdot \langle {:} \tilde B^{\rho_N} (u_s, u_s) {:}, f \rangle_{V^{\ast}, V} \mathd s$ to $\mathcal B^f$, we obtain the improved energy estimate~\eqref{eq:improved-energy} by Lemma~\ref{lem:sufficient-energy-3}. It remains to show that the convergence
    \[
        \int_0^t\cdot \langle {:} \tilde B^{\rho_N} (u_s, u_s) {:}, f \rangle_{V^{\ast}, V} \mathd s \to \mathcal B^f
    \]
    holds in $p$-variation for some $p<2$. By the Besov-variation embedding shown in Corollary A.3 of the appendix of \cite{Friz2010} it suffices to show that, with $I_{s,t} = I_t - I_s$ and for some $p<2$,
    \begin{equation}\label{eq:energy-alternative-pr1}
        \lim_{N\to \infty} \int_{[0,T]^2}\frac{\mathbb E[|I(\langle {:} \tilde B^{\rho_N} (\cdot, \cdot) {:}, f \rangle_{V^{\ast}, V} - \langle {:} B (\cdummy, \cdummy) {:}, f \rangle_{V^{\ast}, V})_{s,t}|^2]}{|t-s|^{1+2/p}}\mathd s \mathd t =0.
    \end{equation}
    Shifting \eqref{eq:improved-energy} to the interval $[s,t]$ by stationarity and ignoring the higher order contribution $(t-s)^2$ because we are interested in short length scales, we obtain from the improved energy estimate and \eqref{eq:B-tilde-speed} for all $M \in \mathbb N$
    \[
        \mathbb E[|I(\langle {:} \tilde B^{\rho_M} (\cdot, \cdot) {:}, f \rangle_{V^{\ast}, V} - \langle {:} B (\cdummy, \cdummy) {:}, f \rangle_{V^{\ast}, V})_{s,t}|^2] \lesssim \varepsilon_M |t-s|.
    \]
    Let $\kappa = \tfrac{1}{1+\alpha}$, where $\|\rho_N\|_{L(H,V)} \lesssim \varepsilon_N^{-\alpha}$. If $\varepsilon_N \leqslant|t-s|^\kappa$, we take $M=N$. If $\varepsilon_N > |t-s|^\kappa$ we take $M>N$ such that $\varepsilon_M \simeq |t-s|^\kappa$ and obtain from \eqref{eq:improved-energy} together with~\eqref{eq:G-rho-bound}, writing $I, I^N, I^M$ for brevity,
    \[
        \mathbb E[|I_{s,t} - I^N_{s,t}|^2] \lesssim \mathbb E[|I_{s,t} - I^M_{s,t}|^2] +  \mathbb E[|I^N_{s,t}|^2] + \mathbb E[|I^M_{s,t}|^2] \lesssim \varepsilon_M |t-s| + \varepsilon_M^{-\alpha}|t-s|^2 \simeq |t-s|^{1+\kappa}.
    \]
    Interpolating this with the bound $\mathbb E[|I_{s,t} - I^N_{s,t}|^2] \lesssim \varepsilon_N |t-s|$, we obtain that \eqref{eq:energy-alternative-pr1} holds for all $p>2/(1+\kappa)$.  
\end{proof}
The proof used the following lemma, which might be useful in its own right.

\begin{lemma}\label{lem:sufficient-energy-3} We make Assumption (A) and additionally assume that $S$ is a core for $A$. Let $(u_t (f))_{t \geqslant 0, f \in  S}$ be an adapted process which satisfies conditions i., ii., iv. from Theorem~\ref{lem:sufficient-energy} and also
\begin{itemize}
    \item[iii.'] For each $f\in S$ there exists a process $\mathcal B^f$ of vanishing quadratic variation such that
    \[ M^f_t \assign u_t (f) - u_0 (f) + \int_0^t u_s(Af)\mathd s - \mathcal B^f_t, \qquad t \geqslant 0, \]
    is a continuous martingale with quadratic variation $[M^f]_t =  2t \langle
    A f, f \rangle_H$.
\end{itemize}
Then $u$ satisfies the improved energy estimate
    \begin{equation}\label{eq:improved-energy}
        \mathbb{E} \left[ \sup_{t \leqslant T} \left| \int_0^t F (u_s) \mathd s
           \right|^2 \right] \lesssim (T + T^2) \| F \|_{\mathcal{H}^{- 1}_0}^2,\qquad F\in \mathcal C.
    \end{equation}
\end{lemma}

\begin{proof}
For $F \in \mathcal{C}$ we obtain from the It{\^o} trick (we can
    apply It{\^o}'s formula for Dirichlet processes by the assumption that the
    drift has vanishing quadratic variation, together with the usual time reversal argument as in~\cite{Gubinelli2018Energy}) that
    \begin{equation}\label{eq:sufficient-energy-3-pr} \mathbb{E} \left[ \sup_{t \leqslant T} \left| \int_0^t \mathcal{L}_0 F
       (u_s) \mathd s \right|^2 \right] \lesssim T \| F \|_{\mathcal{H}^1_0}^2.
    \end{equation}

By Lemma~\ref{lem:core} in the appendix, $\mathcal{C}$ is a core for $\mathcal{L}_0$. Therefore, the bound~\eqref{eq:sufficient-energy-3-pr} extends to all $F \in \mathcal{D} (\mathcal{L}_0)$: Let $(F^N)_N \subset
    \mathcal{C}$ be such that $F^N \rightarrow F$ and $\mathcal{L}_0 F^N
    \rightarrow \mathcal{L}_0 F$. Then also
    \[ \| F - F^N \|_{\mathcal{H}^1_0}^2 \simeq \mathbb{E} [(F - F^N) (1
       -\mathcal{L}_0) (F - F^N)] \leqslant \| F - F^N \| \cdot \| (1
       -\mathcal{L}_0) (F - F^N) \| \rightarrow 0, \]
    and thus by stationarity
    \begin{align*}
      \mathbb{E} \left[ \sup_{t \leqslant T} \left| \int_0^t \mathcal{L}_0 F
      (u_s) \mathd s \right|^2 \right] & =  \lim_{N \rightarrow \infty}
      \mathbb{E} \left[ \sup_{t \leqslant T} \left| \int_0^t \mathcal{L}_0 F^N
      (u_s) \mathd s \right|^2 \right] \lesssim  \lim_{N \rightarrow \infty} T \| F^N
      \|_{\mathcal{H}^1_0}^2 =  T \| F \|_{\mathcal{H}^1_0}^2 .
    \end{align*}

Let now $F \in \mathcal C$ and let $(1 -\mathcal{L}_0) G = F$,
    so $G \in \mathcal{D} (\mathcal{L}_0)$. Then
    \begin{align*}
      \mathbb{E} \left[ \sup_{t \leqslant T} \left| \int_0^t F (u_s) \mathd s
      \right|^2 \right] & \lesssim \mathbb{E} \left[ \sup_{t \leqslant T}
      \left| \int_0^t G (u_s) \mathd s \right|^2 \right] +\mathbb{E} \left[
      \sup_{t \leqslant T} \left| \int_0^t \mathcal{L}_0 G (u_s) \mathd s
      \right|^2 \right]\\
      & \lesssim T^2 \| G \|^2 + T \| G \|_{\mathcal{H}^1_0}^2,
    \end{align*}
    and clearly $\| G \| \leqslant \| G \|_{\mathcal{H}^1_0}$, while $\| G \|_{\mathcal{H}^{1}_0} \leqslant \| F \|_{\mathcal{H}^{- 1}_0}$ because
    \[ \| G \|_{\mathcal{H}^1_0}^2 \simeq \mathbb{E} [G (1 -\mathcal{L}_0) G]
       =\mathbb{E} [G F] \leqslant \| G \|_{\mathcal{H}^1_0} \| F
       \|_{\mathcal{H}^{- 1}_0}. \]
    This proves the claimed energy estimate~\eqref{eq:improved-energy}.
\end{proof}
}

%% file: Chapters/3_Existence_Proofs.tex
\section{Energy solutions and semigroup}\label{sec: Energy solutions and semigroup}

\subsection{Construction of the semigroup and resolvent} \label{sec: Construction of the semigroup and resolvent}

\begin{proposition}
  \label{constructioncandidatesemigroup}Let $\mathcal{G}^N =
  \tmmathbf{1}_{\mathcal{N} \leqslant N} \mathcal{G} \tmmathbf{1}_{\mathcal{N}
  \leqslant N}$. The following statements hold:
  \begin{enumeratenumeric}
    \item \label{existenceresolventN}For every $N{\in}\mathbb{N}$  and $F^{\sharp} \in
    \mathcal{H}^{- 1}_0$ there is a unique solution $F^N \backassign (1 -
    \mathcal{L}_0 - \mathcal{G}^N)^{- 1} F^{\sharp} \in \mathcal{H}^1_0$ to
    the resolvent equation
    \begin{equation}
      (1 - \mathcal{L}_0 - \mathcal{G}^N) F^N = F^{\sharp} .
      \label{resolventequationN}
    \end{equation}
    \item Furthermore, the following bound holds uniformly in $N$
    \begin{equation}
      {\| (1 - \mathcal{L}_0 - \mathcal{G}^N)^{- 1} F^{\sharp}
      \|_{\mathcal{H}_{ 0}^1}}  \lesssim {\| F^{\sharp} \|_{\mathcal{H}_{
      0}^{- 1}}}  .
    \end{equation}
    \item \label{limitresovlentequation}There exists $\mathcal{R}_1 \in L
    (\mathcal{H}^{- 1}_0, \mathcal{H}^1_0)$ such that for every $F^{\sharp}
    \in \mathcal{H}^{- 1}_0$
    \begin{equation}
      (1 - \mathcal{L}_0 - \mathcal{G}^N)^{- 1} F^{\sharp} \rightarrow
      \mathcal{R}_1 F^{\sharp},
    \end{equation}
    weakly in $\mathcal{H}^1_0$ and $F = \mathcal{R}_1 F^{\sharp} \in
    \mathcal{H}^1_0$ is a solution to the resolvent equation
    \[ (1 - \mathcal{L}_0 - \mathcal{G}) F = F^{\sharp} . \]
    \item With $\mathcal{D} = \mathcal{R}_1 L^2 (\mu) \subset
    \mathcal{D}_{\max} \assign \{ F \in \mathcal{H}^1_0 : \mathcal{L} F \in
    L^2 (\mu) \} \subset \mathcal{H}^1_0$ it holds that $(\mathcal{D},
    \mathcal{L})$ generates a strongly continous contraction semigroup
    $(P_t)_{t \geqslant 0}$ on $L^2 (\mu)$.
  \end{enumeratenumeric}
\end{proposition}

\begin{proof}
  \begin{enumeratenumeric}
    \item We want to apply the Lax-Milgram theorem, see \cite[sec.~6.21]{evans2022partial}. For this purpose, we consider the bilinear
    operator $K^N : \mathcal{H}^1_0 \times \mathcal{H}^1_0 \rightarrow
    \mathbb{R}$ given by
    \begin{equation}
      K^N (\cdot, \cdot) \assign \langle (1 -\mathcal{L}_0 -\mathcal{G}^N)
      \cdot, \cdot \rangle_{ \mathcal{H}^{- 1}_0,\mathcal{H}^1_0} .
    \end{equation}
    By Lemmas~\ref{lem:L0-bounds} and~\ref{lem:G-bounds} (bounds on
    $\mathcal{L}_0$ and $\mathcal{G}$) the operator $(1 -\mathcal{L}_0
    -\mathcal{G}^N)$ continuously maps $\mathcal{H}^1_0$ to $\mathcal{H}^{-
    1}_0$ and
    \begin{align}
      K^N (F, G) & =  \langle (1 -\mathcal{L}_0) F, \mathcal{} G
      \rangle_{\mathcal{H}^1_0, \mathcal{H}^{- 1}_0 } - \langle \mathcal{G}^N
      F, \mathcal{} G \rangle_{\mathcal{H}^1_0, \mathcal{H}^{- 1}_0}
      \nonumber\\
      & \leqslant  \| F \|_{\mathcal{H}^1_0} \| G \|_{\mathcal{H}^1_0}
      +\|\mathcal{G}^N F\|_{\mathcal{H}^1_0} \| \mathcal{} G\|_{\mathcal{H}^{-
      1}_0 } \nonumber\\
      & \lesssim_N  (\| F \|_{\mathcal{H}^1_0} + \|
      \tmmathbf{1}_{\mathcal{N} \leqslant n} F \|_{\mathcal{H}^1_0}) \| G
      \|_{\mathcal{H}^1_0} \nonumber\\
      & \lesssim  \| F \|_{\mathcal{H}^1_0} {\| G \|_{\mathcal{H}^1_0}} .
    \end{align}
    Moreover, by Lemma~\ref{lem:L0-bounds} (bounds on $\mathcal{L}_0$) and the
    fact that $\mathcal{G}^N$ is antisymmetric, we have for every $F \in
    \mathcal{H}^1_0  \mathcal{}$
    \begin{align}
      K^N (F, F)  & =  \langle (1 -\mathcal{L}_0 -\mathcal{G}^N) F,
      \mathcal{} F \rangle_{\mathcal{H}^1_0, \mathcal{H}^{- 1}_0} \nonumber\\
      & =  \langle (1 -\mathcal{L}_0) \mathcal{} F, F
      \rangle_{\mathcal{H}^1_0, \mathcal{H}^{- 1}_0} \nonumber\\
      & \simeq  \| F \|_{\mathcal{H}^1_0}^2 .  \label{prepbilinear}
    \end{align}
    Therefore, $K^N$ satisfies the conditions of the Lax-Milgram theorem, and
    given $F^{\sharp} \in \mathcal{H}^{- 1}_0$, there exists a unique $F^N \in
    \mathcal{H}^1_0$ such that
    \begin{equation}
      K^N (F^N, G) = \langle F^{\sharp}, G \rangle_{\mathcal{H}^{- 1}_0,
      \mathcal{H}^1_0}, \qquad G \in \mathcal{H}^1_0 .
    \end{equation}
    This concludes the proof of \ref{existenceresolventN}.
    
    \item By eq. (\ref{prepbilinear}) we have
    \begin{equation}
      \| F^{\sharp} \|_{\mathcal{H}^{- 1}_0} \| F^N \|_{\mathcal{H}^1_0}
      \geqslant \langle F^{\sharp}, F^N \rangle_{\mathcal{H}^{- 1}_0,
      \mathcal{H}^1_0} = K^N (F^N, F^N ) \simeq \| F^N \|_{\mathcal{H}^1_0}^2,
    \end{equation}
    uniformly in $N$, so that
    \begin{equation}
      \| F^N \|_{\mathcal{H}^1_0} \lesssim \| F^{\sharp} \|_{\mathcal{H}^{-
      1}_0} .
    \end{equation}
    \item With the weak convergence, this follows by testing against $G \in
    \mathcal{C}$ and applying the dominated convergence theorem and using that
    $G \in \mathcal{H}^1_1$ to obtain $\| (\mathcal{G} - \mathcal{G}^N) G
    \|_{\mathcal{H}^{- 1}_0} \rightarrow 0$. See
    \cite[Prop.~1.2]{perkowski2024energy} or \cite[Prop.~4.9]{grafner2023energy} for details.
    
    \item This is due to~\cite{Graefner2024SPDE} and it follows from the Lumer-Pillips theorem. See \cite[Prop.~1.2]{perkowski2024energy} or \cite[Prop.~4.9]{grafner2023energy}
    for details. \qedhere
  \end{enumeratenumeric}
\end{proof}

\begin{lemma}
  \label{corepropertycyl}It holds that
  \[ \langle (1 -\mathcal{L}) F, F \rangle_{L^2 (\mu)} \simeq \| F
     \|_{\mathcal{H}^1_0}^2, \qquad F \in \mathcal{D}_{\max} = \{ F \in
     \mathcal{H}^1_0 : \mathcal{L} F \in L^2 (\mu) \} \subset \mathcal{H}^1_0
     . \]
  Consequently, $(1 -\mathcal{L})$ is injective on $\mathcal{D}_{\max}$,
  \[ (1 - \mathcal{L}) \mathcal{C} \subset \mathcal{H}^{- 1}_0, \]
  is dense in $\mathcal{H}^{- 1}_0$ and the domain $\mathcal{D}$ from the
  previous Proposition \ref{constructioncandidatesemigroup} is equal to the
  maximal domain $\mathcal{D}_{\max}$.
\end{lemma}

\begin{proof}
  The proof is based on the commutator estimate (\ref{commutatorestimate1})
  and can be found in e.g.
  \cite[Thm.~1.4]{perkowski2024energy} or \cite[Thm.~4.11]{grafner2023energy}.
\end{proof}

\subsection{Construction of energy solutions}\label{sec: Construction of energy solutions}
In this section we assume that there exists a complete orthonormal system
$(e_n)_{n \in \mathbb{N}} \subset V$ with respect to $H$ of eigenvectors for $A$ with an increasing sequence of eigenvalues $(\lambda_n)_{n \in \mathbb{N}}$, which morally corresponds to the case that $V$ is compactly embedded in $H$. This assumption could be bypassed by a more complex approximation argument, involving a natural discretization of the spectrum of $A$ that leads to an $H$-bounded perturbation of this operator, but for simplicity we restrict our attention to the special case where $A$ can be diagonalized. Note that $\lambda_n \geqslant 0$ for all $n \in \mathbb{N}$ because $\langle A f, f \rangle_H
\geqslant c \| f \|_{\dot{V}}^2$. 
We also assume that $ S$ is a nuclear Fréchet space.

Let $N \in \mathbb{N}$. Consider the projected SDE corresponding to equation (\ref{eqn}), for $k = 1, \ldots, N$,
\begin{equation}\label{eq:SDE} \mathd u^N_t (e_k) = \left( - \lambda_k u^N_t (e_k) + \langle {:}B^N(u^N_t,u^N_t){:}, e_k\rangle_{V^\ast, V} \right) \mathd t + \sqrt{2
   \lambda_k} \mathd W^k_t, \end{equation}
where $(W^k)_k$ are i.i.d. Brownian motions and
\[
  {:}B^N (u, u){:} \assign \sum_{k, \ell, m \leqslant N} (u (e_{\ell}) u (e_m) -
     \delta_{\ell, m}) \langle B (e_{\ell}, e_m), e_k \rangle_{V^{\ast}, V} e_k. 
\]
We claim that an invariant probability measure is
\[
    (u^N_k)^N_{k = 1} : = (u^N (e_k))_{k = 1}^N \sim \mu_N =\mathcal{N} (0,\mathbb{I}_{N \times N}) .
\]
Indeed, the generator of the SDE is
\[ \mathcal{L}^N = \sum_{k = 1}^N \left( \lambda_k (- u^N_k
   \partial_k + \partial_{k k}) + \sum_{\ell, m \leqslant N} (u^N_{\ell} u^N_m
   - \delta_{\ell, m}) \langle B (e_{\ell}, e_m), e_k \rangle_{V^{\ast}, V}
   \partial_k \right) . \]
It can be split up into a symmetric part
\[ \mathcal{L}_0^N = \sum_{k = 1}^N \lambda_k (- u^N_k
   \partial_k + \partial_{k k}), \]
which is the generator of an Ornstein-Uhlenbeck process in $N$ dimensions and
hence preserves $\mu_N$, whereas by analogous reasoning as in Lemma \ref{Gantisymmetric} the contribution of the drift
\[ \mathcal{G}^N = \sum_{k, \ell, m \leqslant N} (u^N_{\ell}
   u^N_m - \delta_{\ell, m})  \langle B (e_{\ell}, e_m), e_k \rangle_{V^{\ast},
   V} \partial_k, \]
is antisymmetric and hence the full generator $\mathcal{L}^N$ has $\mu_N$ as
an invariant measure. We start $u^N$ in its invariant measure (and discuss relaxed assumptions later), and we test against $f \in H$ via
\[ u^N_t (f) = \sum_{k = 1}^N u^N_t (e_k) \langle e_k, f \rangle_H . \]

Let us first derive an energy estimate for $u^N$, uniformly in $N$. We define the projection
\begin{equation}
    \Pi_N h = \sum_{k=1}^N \langle h,e_k\rangle_H e_k,
\end{equation}
and for a cylinder function $F(\eta)=\Phi(\eta(f))$ we define a new cylinder function $\Pi_N F(\eta) = \Phi(\eta(\Pi_N f))$.

\begin{lemma}[It{\^o} trick and energy estimate]
  For any $p \geqslant 2$ the following ``It\^o trick'' bound holds uniformly in $N$
  \begin{equation}\label{eq:itotrick-existence} \mathbb{E} \left[ \sup_{t \leqslant T} \left| \int_0^t \mathcal{L}_0^N F
     (u_s^N) \mathd s \right|^p \right] \lesssim T^{p / 2} \left\| \sum_{k =
     1}^N \lambda_k | \partial_k F |^2 \right\|_{L^{p / 2} (\mu_N)}^{p / 2},
     \qquad F \in \mathcal{C},
  \end{equation}
  and if $p = 2$ or if $F(\eta) = \Phi(\eta(f))$ with a polynomial $\Phi$ of degree $M$, then we also have the \emph{energy estimate}
  \begin{equation}\label{eq:energy-estimate-existence} \mathbb{E} \left[ \sup_{t \leqslant T} \left| \int_0^t F (u_s^N) \mathd s
     \right|^p \right] \lesssim (T^{p / 2} + T^p) \| \Pi_N F \|_{\mathcal{H}^{-
     1}_0}^p,
  \end{equation}
  with an implicit constant that depends on $M$ if $p > 2$.
\end{lemma}

\begin{proof}
  We assume without loss of generality that $F=\Pi_N F$; otherwise we use that $F=\Pi_N F$ $\mu^N$-a.s. and derive the estimate for $\Pi_N F$ in place of $F$. \textcolor{red}{By time reversal of Markov processes, $u^N$ is an energy solution for the projected equation \eqref{eq:SDE} with non-linearity $B^N$ in the sense of Lemma \ref{lem:sufficient-energy}}. Therefore, the same arguments as in~\cite[Proposition 3.2]{Gubinelli2018Energy} show that for any cylinder function $F$
  \[ 2 \int_0^t \mathcal{L}_0^N F (u_s^N) \mathd s = M_t^F + \hat{M}^F_{T - t}
     - \hat{M}^F_T, \]
  where $\mathd M^F_t = \sum_{k = 1}^N \partial_k F (u^N_t) \sqrt{2 \lambda_k} \mathd W^k_t$
  and thus
  \( \mathd [M^F]_t = 2 \sum_{k = 1}^N \lambda_k | \partial_k F (u^N_t) |^2, \)
  and similarly for the backward martingale $\hat{M}^F$. The Burkholder-Davis-Gundy inequality now yields
  \[ \mathbb{E} \left[ \sup_{t \leqslant T} \left| \int_0^t \mathcal{L}_0^N F
     (u_s^N) \mathd s \right|^p \right] \lesssim T^{p / 2} \left\| \sum_{k =
     1}^N \lambda_k | \partial_k F |^2 \right\|_{L^{p / 2} (\mu_N)}^{p / 2}, \]
  and for $p = 2$ the right hand side is
  \[ T \sum_{k = 1}^N \lambda_k \| \partial_k F \|_{L^2 (\mu_N)}^2 = T \langle
     F, (-\mathcal{L}_0^N) F \rangle_{L^2 (\mu^N)} . \]
  If $\Phi$ is a polynomial of degree $M$, then
  $\sum_{k = 1}^N \lambda_k | \partial_k F |^2$ is a polynomial of degree $2 M - 2$ and thus we obtain from Gaussian hypercontractivity
  \[ \left\| \sum_{k = 1}^N \lambda_k | \partial_k F |^2 \right\|_{L^{p / 2}
     (\mu_N)}^{p / 2} \simeq_M \langle F, (-\mathcal{L}_0^N) F \rangle_{L^2
     (\mu^N)}^{p/2}. \]
  We can of course also bound $\mathbb{E} \left[ \sup_{t \leqslant T} \left| \int_0^t F (u_s^N) \mathd s \right|^p \right] \lesssim T^p \| F \|_{L^p (\mu_N)}^p$ with the triangle inequality,
  and this means that for $p = 2$ or in the hypercontractive case
  \[ \mathbb{E} \left[ \sup_{t \leqslant T} \left| \int_0^t (1
     -\mathcal{L}_0^N) F (u_s^N) \mathd s \right|^p \right] \lesssim (T^{p /
     2} + T^p) \| (1 -\mathcal{L}_0^N)^{1 / 2} F \|_{L^2 (\mu^N)}^p, \]
  or, noting that we can invert $(1 -\mathcal{L}^N_0)$ in $L^2(\mu^N)$,
  \begin{align*}
    \mathbb{E} \left[ \sup_{t \leqslant T} \left| \int_0^t F (u_s^N) \mathd s
    \right|^p \right] & \lesssim  (T^{p / 2} + T^p) \| (1
    -\mathcal{L}_0^N)^{1 / 2} (1 -\mathcal{L}_0^N)^{- 1} F \|_{L^2
    (\mu^N)}^p\\
    & =  (T^{p / 2} + T^p) \| (1 -\mathcal{L}_0^N)^{- 1 / 2} F \|_{L^2
     (\mu^N)}^p
  \end{align*}
  By duality and density of $\Pi_N \mathcal C$ in $L^2(\mu^N)$, we have for $\eta^N\sim \mu^N = \mu\circ \Pi_N^{-1}$
\[ \| (1 -\mathcal{L}_0^N)^{- 1 / 2} F \|_{L^2 (\mu^N)} = \sup \{ \mathbb{E}
   [F (\eta^N) G (\eta^N)] : G = \Pi_N G \in \mathcal{C}, \| (1 -\mathcal{L}_0^N)^{1 /
   2} G \|_{L^2 (\mu^N)} \leqslant 1 \}. \]
We use that
$\mathcal{L}_0^N G = \mathcal{L}_0 G = \Pi_N \mathcal{L}_0 G$ because $e_k$ are eigenfunctions of
$A$ to obtain
\[ \| (1 -\mathcal{L}_0^N)^{1 / 2} G \|_{L^2 (\mu^N)} = \| (1
   -\mathcal{L}_0)^{1 / 2} G \|_{L^2 (\mu)}, \]
and since $F=\Pi_N F$ and $G=\Pi_N G$, we end up with
\begin{align*}
  \| (1 -\mathcal{L}_0^N)^{- 1 / 2} F \|_{L^2 (\mu^N)} & =  \sup \{
  \mathbb{E} [F (\eta^N) G (\eta^N)] : G \in \mathcal{C}, G = \Pi_N G, \| (1
  -\mathcal{L}_0)^{1 / 2} G \|_{L^2 (\mu)} \leqslant 1 \}\\
  & \leqslant  \sup \{ \mathbb{E} [ F (\eta) G (\eta)] : G \in
  \mathcal{C}, \| (1 -\mathcal{L}_0)^{1 / 2} G \|_{L^2 (\mu)} \leqslant 1 \}\\
  & =  \| (1 -\mathcal{L}_0)^{- 1 / 2}  F \|_{L^2 (\mu)} . \qedhere
\end{align*}
\end{proof}

\begin{lemma}\label{lem:tightness}
  For all $f \in V$ the processes $(u^N (f))_{N \in \mathbb{N}}$ are uniformly
  tight in $C(\mathbb R_+, \mathbb R)$, and there exists $p<2$ such that also the real-valued random variables $(\|\int_0^\cdot
    \langle {:}B^N (u^N_r, u^N_r){:}, f \rangle_{V^{\ast}, V} \mathd r \|_{[0,T],p-var})_{N}$ are uniformly tight for all $T>0$.
\end{lemma}

\begin{proof}
  Tightness at the initial time is clear, because
  \[ \mathbb{E} [u^N_0 (f)^2] = \| \Pi_N f \|_H^2 \leqslant \| f \|_H^2 . \]
  For the time increment $u^N_{s,t}(f) = u^N_t(f) - u^N_s(f)$ we obtain
  \begin{align*}
    u^N_{s, t} (f) & =  \int_s^t u^N_r (- A \Pi_N f) \mathd r + \int_s^t
    \langle {:}B^N (u^N_r, u^N_r){:}, f \rangle_{V^{\ast}, V} \mathd r + M^{f,
    N}_{s, t}.
  \end{align*}
  The quadratic variation of $M^{f,N}$ is
  \[ [M^{f, N}]_{s,t} = 2 (t-s) \sum_{k = 1}^N \lambda_k \langle f, e_k \rangle_H^2 \leqslant 2 (t-s) \sum_{k = 1}^\infty \lambda_k \langle f, e_k \rangle_H^2 =
     2 (t-s) \langle A f, f \rangle_H \simeq (t-s) \| f
     \|_{\dot{V}}^2, \]
  which gives tightness for the martingale. Also, $u (A \Pi_N f) =\mathcal{L}_0^N u (f)$ and thus by the It{\^o} trick~\eqref{eq:itotrick-existence}
  \[
    \mathbb{E} \left[ \left| \int_s^t u^N_r (- A \Pi_N f) \mathd r \right|^p
    \right] \lesssim | t - s |^{p / 2} \sum_{k = 1}^N \lambda_k \langle f, e_k \rangle_H^2 \leqslant | t - s |^{p / 2} \| f \|_{\dot{V}}^2,
  \]
  which yields tightness for the linear contribution to the drift. It remains to treat the nonlinear part of the drift. Note that $\langle {:}B^N(\eta,\eta){:}, f\rangle = \mathcal G^{\Pi_N} \eta(f)$ is a second order polynomial in $\eta$, and therefore the hypercontractive version of the energy estimate, \eqref{eq:energy-estimate-existence}, together with the bounds for $\mathcal{G}^{\Pi_N}$ from~\eqref{eq:G-bounds} yields
  \begin{align}
     \mathbb{E}\left[\left|\int_s^t\langle{:}B^N(u^N_r,u^N_r){:},\Pi_Nf\rangle_H\mathd r\right|^p\right] & \lesssim |t-s|^{p/2}\left\|\mathcal G^{\Pi_N} \eta(f)\right\|_{\mathcal{H}^{-1}_0}^2 \nonumber \\
     & \lesssim |t-s|^{p/2}(\|
    \Pi_N \|_{L (H, H)} + \| \Pi_N \|_{L (\dot{V}, \dot{V})})^3 \| \eta(f)  \|_{\mathcal{H}^1_{1}} \nonumber \\
     & \simeq |t-s|^{p/2}(\|
    \Pi_N \|_{L (H, H)} + \| \Pi_N \|_{L (\dot{V}, \dot{V})})^3 \| f  \|_{V},\label{eq:tightness-pr1}
  \end{align}
uniformly in $N$. Thus, we obtain tightness as soon as we show that $\|\Pi_N \|_{L (H, H)} + \| \Pi_N \|_{L (\dot{V}, \dot{V})}\lesssim 1$. The first term is clearly bounded by $1$ because $\Pi_N$ is an orthogonal projection in $H$. For the second term we use that $(e_k)$ are eigenvectors for $A$ to bound
  \[
     \| \Pi_N f \|_{\dot V}^2 \simeq \langle A \Pi_N f, \Pi_N f\rangle_{V^\ast, V} \leqslant \langle A f, f\rangle_{V^\ast, V} \simeq \| f\|_{\dot V}^2,
  \]
  and this concludes the proof of tightness of $(u^N)$. Tightness for the $p$-variation is shown with the same argument as in Lemma~\ref{lem:sufficient-energy}.
\end{proof}

\begin{theorem}[Existence of energy solutions]\label{thm:exisence}
We make Assumption~\ref{ass:A} and~\ref{ass:B} and we assume that there exists an orthonormal basis of eigenvectors for $A$ and that $ S$ is a nuclear Fr\'echet space. 
Let $\nu \ll \mu$ be such that $\frac{\mathd\nu}{\mathd\mu} \in L^2(\mu)$.
Then there exists an energy solution $u$ with $u_0 \sim \nu$.
\end{theorem}

\begin{proof}
Let $\mathcal F_N = \sigma(\eta(e_1), \dots, \eta(e_N))$ and let $\gamma^N = \E[\mathd \nu / \mathd \mu|\mathcal F_N]$ (conditional expectation in $L^2(\mu)$) and $\mathd \nu^N = \gamma^N \mathd \mu^N$. Let $u^N$ be a solution to the SDE~\eqref{eq:SDE} with $u^N_0\sim \nu^N$. We claim that the sequence $(u^N)_{N \in \mathbb N}$ is uniformly tight in $C(\mathbb R_+,  S')$ and that any limit point $u$ is an energy solution with $u_0 \sim \nu$.

To show tightness of $(u^N)_N$ in $C(\mathbb R_+,  S')$ we apply Mitoma's criterion~\cite{Mitoma1983}, which states that it suffices to show tightness of $(u^N(f))_N$ in $C(\mathbb R_+, \mathbb R)$ for all $f \in  S$; this is the only part of the argument where we use that $ S$ is a nuclear Fr\'echet space. Lemma~\ref{lem:tightness} gives us the tightness of $(u^N(f))_N$ if $\nu=\mu$. Otherwise, we perform a change of measure at the initial time and we apply the Cauchy-Schwarz inequality and Jensen's inequality: For any bounded and measurable $\Psi: C(\mathbb R_+, \mathbb R) \to \R$
\begin{align}
    \E_{\nu^N}[\Psi(u^N)] & = \int \frac{\mathd \nu^N}{\mathd\mu^N}(u) \E_{u}[\Psi(u^N)]  \mu^N(\mathd u) \leqslant \left\| \frac{\mathd\nu^N}{\mathd\mu^N} \right\|_{L^2(\mu^N)} \left(\int  \E_{u}[\Psi(u^N)]^2  \mu^N(\mathd u)\right)^{1/2} \nonumber \\
    &\leqslant \left\| \frac{\mathd\nu}{\mathd\mu} \right\|_{L^2(\mu)} \E_{\mu^N}[\Psi(u^N)^2]^{1/2}.\label{eq:existence-pr1}
\end{align}
Tightness follows by taking $\Psi$ as the indicator function of the complement of a compact set $K$ in $C(\mathbb R_+, \mathbb R)$ such that $\sup_N \P_{\mu^N}(u^N(f) \in K^c) \leqslant \varepsilon$. 

Next, we fix a limit point $u$ along a subsequence that we denote for simplicity again with $(u_N)$. We need to show that $u$ is an energy solution. Taking $\Psi(u) = |F(u_t)|$ in \eqref{eq:existence-pr1}, we obtain the uniform incompressibility estimate
\begin{align*}
    \E_{\nu^N}[|F(u^N_t)|] \leqslant \left\| \frac{\mathd\nu}{\mathd\mu} \right\|_{L^2(\mu)} \left\| \Pi_N F \right\|_{L^2(\mu)}.
\end{align*}
Taking $\Psi(u) = \sup_{t \leqslant T}| \int_0^t F(u_s)\mathd s|$ and combining \eqref{eq:existence-pr1} with the energy estimate~\eqref{eq:energy-estimate-existence} gives
\[
    \E_{\nu^N}\left[\sup_{t \leqslant T}\left| \int_0^t F(u^N_s) \mathd s\right|\right] \lesssim (T^{1/2}+T) \|\Pi_N F\|_{\mathcal H^{-1}_0}.
\]
By Fatou's lemma, $u$ satisfies points i. (path regularity), ii. (incompressibility), iii. (energy estimate) in Definition~\ref{def:energy} of energy solutions, provided that in ii. and iii. we restrict to cylinder functions satisfying $F = \Pi_N F$ for some $N \in \N$. Since such ``finite range'' cylinder functions are closed under pointwise multiplication, the incompressibility estimate ii. for general $F \in \mathcal C$ follows by a monotone class argument. The energy estimate iii. for general $F \in \mathcal C$ is slightly more subtle. We let $F_N = \E_\mu[F|\mathcal F_N]$ with $F_N = \Pi_N F_N$ and apply the incompressibility condition ii. to obtain
\[
    \E\left[\sup_{t \leqslant T}\left| \int_0^t F(u_s) \mathd s\right|\right] = \lim_{N\to \infty} \E\left[\sup_{t \leqslant T}\left| \int_0^t F_N(u_s) \mathd s\right|\right] \lesssim \limsup_{N\to \infty} (T^{1/2}+T) \|F_N\|_{\mathcal H^{-1}_0}.
\]
Now for $G \in \mathcal C$:
\[
    \E[F_N G]=\E[F\E[G|\mathcal F_N]] \leqslant \|F\|_{\mathcal H^{-1}_0} \|\E[G|\mathcal F_N]\|_{\mathcal H^1_0},
\]
so the energy estimate follows if we can show that $G\mapsto \E[G|\mathcal F_N]$ is bounded linear map on $\mathcal H^1_0$. Since it is obviously bounded $L^2(\mu)$, it suffices to show boundedness in $\dot{\mathcal H}^1_0$:
\[
    \|\E[G|\mathcal F_N]\|_{\dot{\mathcal H}^1_0}^2 \simeq \sum_k \lambda_k \E[|D_{e_k}\E[G|\mathcal F_N]|^2] = \sum_{k\leqslant N} \lambda_k \E[|\E[D_{e_k}G|\mathcal F_N]|^2]\leqslant \sum_k \lambda_k \mathbb E[|D_{e_k}G|^2] \simeq \|G\|_{\dot{\mathcal H}^1_0}^2.
\]
This proves the energy estimate for general $F \in \mathcal C$.

It remains to show that any limit point satisfies the weak solution property (iv.). Let $f \in  S$, and note that
\[
    u^N_t(f) = u^N_0(f) - \int_0^t u^N_s(Af)\mathd s + \int_0^t \langle {:}B^N(u^N_s,u^N_s){:}, f\rangle_{V^\ast, V}\mathd s + M^{f,N}_t.
\]
The martingale and the initial condition converge to the desired limit by the usual arguments, see e.g. \cite{Goncalves2014, Gubinelli2013}. For the linear drift there is a small problem because we do not know if $Af \in  S$, and therefore we cannot use the the weak convergence in $C(\R_+, \R)$ to describe the limit of $\int_0^\cdot u^N_s(Af)\mathd s$. But by density of $ S$ in $H$ and by the uniform incompressibility estimate we can approximate $\int_0^\cdot u^N_s(Af)\mathd s$ by $\int_0^\cdot u^N_s(g)\mathd s$ with $g \in  S$ up to making an $L^1$-error of at most $\varepsilon$, pass to the limit $\int_0^\cdot u_s(g)\mathd s$, and replace this by $\int_0^\cdot u_s(Af)\mathd s$, again making an $L^1$-error of at most $\varepsilon$.

We are left with deriving the limit of the quadratic drift, and with showing that it has the required approximation property in $p$-variation. Let  $\mathcal{B}^f$ be a limit point of $\mathcal{B}^{f,N} = \int_0^{\cdot} \langle {:}
B^N (u^N_s, u^N_s) {:}, f \rangle_{V^{\ast}, V} \mathd s$, which
exists by tightness of this sequence, let $p' < 2$ be such that $(
\| \mathcal{B}^{f,N} \|_{[0, T], p' - \tmop{var}})_N$ is tight (see
Lemma~\ref{lem:tightness}), and let $p \in (p', 2)$. By density of $\bigcup_K
\Pi_K \mathcal{C}$, we can find for each $M \in \N$ a cylinder function $F_{M, f} =
\Pi_{K_M} F_{M, f}$ with $\| \langle {:} \tilde{B}^{\rho_M} (\cdummy, \cdummy)
{:}, f \rangle_{V^{\ast}, V} - F_{M, f} \|_{L^2 (\mu)} \leqslant \frac{1}{M}$,
and as $\left\| \int_0^{\cdot} (\cdummy) \mathd s \right\|_{[0, T], p
- \mathrm{var}} \leqslant \int_0^T | (\cdummy) | \mathd s$, this yields
\[
    \mathbb{E} \left[ \left\| \mathcal{B}^f -  \int_0^\cdot\langle {:} \tilde{B}^{\rho_M} (u_s, u_s){:}, f \rangle_{V^{\ast}, V} \mathd s \right\|_{[0, T], p - \mathrm{var}} \right] \leqslant \mathbb{E} \left[\left\| \mathcal{B}^f - \int_0^{\cdot} F_{M, f} (u_s) \mathd s \right\|_{[0, T], p - \mathrm{var}} \right] + \frac{T}{M} .
\]
By an interpolation argument, tightness in $p$-variation uniform topology is
equivalent to tightness in the uniform topology together with tightness of the
$p'$-variation norm, see \cite[Thm.~6.1]{Friz2018}.
Together with the convergence $\int_0^{\cdot} F_{M, f} (u_s^N) \mathd s
\rightarrow \int_0^{\cdot} F_{M, f} (u_s) \mathd s$ in uniform $1$-variation
topology, we thus obtain with Fatou's lemma for some sufficiently large $p<2$ (whose value will be fixed at the end of the proof),
\begin{align}
  & \mathbb{E} \left[ \left\| \mathcal{B}^f - \int_0^{\cdot} F_{M, f} (u_s)
  \mathd s \right\|_{[0, T], p - \mathrm{var}} \right]\nonumber \\
  & \leqslant \liminf_{N \rightarrow \infty} \mathbb{E} \left[ \left\|
  \int_0^{\cdot} \langle {:} B^N (u^N_s, u^N_s) {:}, f \rangle_{V^{\ast}, V}
  \mathd s - \int_0^{\cdot} F_{M, f} (u_s^N) \mathd s \right\|_{[0, T], p -
  \mathrm{var}} \right]\nonumber \\
  & \lesssim \liminf_{N \rightarrow \infty} \mathbb{E}_{\mu^N} \left[ \left\|
  \int_0^{\cdot} \langle {:} B^N (u^N_s, u^N_s) {:}, f \rangle_{V^{\ast}, V}
  \mathd s - \int_0^{\cdot} \langle {:} \tilde{B}^{\rho_M} (u^N_s, u^N_s) {:}, f
  \rangle_{V^{\ast}, V} \mathd s \right\|_{[0, T], p - \mathrm{var}}^2 \right]^{1/2} +
  \frac{T}{M} \nonumber \\
  & \lesssim \varepsilon_M^{\beta} + \frac{T}{M} \label{eq:existence-pr5}
\end{align}
for some $\beta > 0$, where the last step will be justified in the remainder of this proof. Once~\eqref{eq:existence-pr5} is established, we have shown that $\mathbb{E} [ \| \mathcal{B}^f -  \int_0^\cdot\langle {:} \tilde{B}^{\rho_M} (u_s, u_s){:}, f \rangle_{V^{\ast}, V} \mathd s\|_{[0, T], p - \mathrm{var}} ] $ converges to $0$ as $M\to \infty$ and this concludes the proof. To show \eqref{eq:existence-pr5} we use that, since $u^N = \Pi_N u^N$,
\begin{align}
    &\langle {:}B^N(u^N_s,u^N_s){:}, f\rangle_{V^\ast, V}  - \langle {:}\tilde B^{\rho_M}(u^N_s,u^N_s){:}, f\rangle_{V^\ast, V}  = \langle {:}B(\Pi_N u^N_s,\Pi_N u^N_s){:}, \Pi_N f - f\rangle_{V^\ast, V} \nonumber \\
    & \hspace{80pt} + \langle {:}B(\Pi_N u^N_s,\Pi_N u^N_s){:} -{:}\tilde B^{\rho_M}(\Pi_N u^N_s,\Pi_N u^N_s){:},  f\rangle_{V^\ast, V}. \label{eq:existence-pr3}
\end{align}
For the first term on the right hand side we simply bound with the Cauchy-Schwarz inequality
\begin{align*}
    \E_{\mu^N}\left[ \left|\int_s^t \langle {:}B(\Pi_N u^N_r,\Pi_N u^N_r){:}, \Pi_N f - f\rangle_{V^\ast, V} \mathd r\right|^2\right] & \lesssim (t-s)^2 \|\Pi_N\|_{L(H, V)} \|\Pi_N\|_{L(H, H)} \| \Pi_N f - f\|_{\dot V}\\
    & \lesssim (t-s)^2 (1+\lambda_N)^{1/2}\lambda_N^{-1/2} \|(1-\Pi_N)Af\|_H \\
    & \lesssim (t-s)^2 \|(1-\Pi_N)Af\|_H,
\end{align*}
where we used that
\[ \|\Pi_N h \|_{V}^2 \simeq \sum_{k \leqslant N} (1+ \lambda_k) \langle h, e_k\rangle_H^2 \leqslant (1+\lambda_N) \|h\|_H^2 \]
and
\[
    \| \Pi_N f - f\|_{\dot V}^2 \simeq \sum_{k>N} \lambda_k \langle f, e_k\rangle_H^2 \leqslant \lambda_N^{-1}\sum_{k>N} \lambda_k^2 \langle f, e_k\rangle_H^2 = \lambda_N^{-1} \| (1 - \Pi_N) Af\|_H^2.
\]
By the Besov-variation embedding already used in Lemma~\ref{lem:sufficient-energy}, we get for any $p<2$
\begin{equation} \label{eq:existence-pr4}
  \lim_{N \rightarrow \infty} \mathbb{E}_{\mu^N} \left[ \left\| \int_0^{\cdot} \langle {:}B(\Pi_N u^N_s,\Pi_N u^N_s){:}, \Pi_N f - f\rangle_{V^\ast, V} \mathd s \right\|_{[0, T], p - \mathrm{var}}^2 \right]^{1/2} = 0.
\end{equation}
For the remaining term in~\eqref{eq:existence-pr3} we use the same interpolation argument as in the proof of Lemma~\ref{lem:sufficient-energy} together with the bound $\|\Pi_N^{\otimes 2}\|_{L(V\otimes_s H, V\otimes_s H)} \lesssim 1$ from Lemma~\ref{lem:tensor-operator} in the appendix, to obtain a bound on the increment:
\begin{align}\label{eq:existence-pr2}
    \E_{\mu^N}\left[ \left|\int_s^t \langle {:}B(\Pi_N u^N_r,\Pi_N u^N_r){:} -{:}\tilde B^{\rho_M}(\Pi_N u^N_r,\Pi_N u^N_r){:},  f\rangle_{V^\ast, V} \mathd r\right|^2\right] \lesssim \varepsilon_M^\beta (t-s)^{1+\zeta}
\end{align}
for some $\beta,\zeta>0$. If $p<2$ is such that $2/p < 1+\zeta$ and such that there exists $p'\in(p,2)$ for which $(
\| \mathcal{B}^{f,N} \|_{[0, T], p' - \tmop{var}})_N$ is tight, we combine~\eqref{eq:existence-pr3}, \eqref{eq:existence-pr4}, \eqref{eq:existence-pr2} and we use the same Besov-variation embedding as in Lemma~\ref{lem:sufficient-energy} to deduce~\eqref{eq:existence-pr5}.
\end{proof}

%% file: Chapters/4_Uniqueness_Proof.tex
\subsection{Uniqueness proof: Duality of semigroup and energy solutions}\label{sec: Uniqueness proof: Duality of semigroup and energy solutions}

To prove the uniqueness of energy solutions, we will show that every energy solution solves the martingale problem for the operator $(\mathcal D_{\mathrm{max}}, \mathcal L)$ constructed in Proposition~\ref{constructioncandidatesemigroup} and Lemma~\ref{corepropertycyl}. For that purpose, we first extend the weak solution property from linear test functions to cylinder functions.

\begin{lemma}[It{\^o} formula]\label{lem:Ito}
  Let $u$ be an energy solution and let $F \in \mathcal{C}$. Then 
  \[ F(u_t) - F(u_0) - I (\mathcal{L}F)_t, \qquad t \geqslant 0, \]
  is a continuous martingale.
\end{lemma}

\begin{proof} 
   This follows by the same arguments as \cite[Lem.~4.14]{Gubinelli2020}, relying on It\^o's formula, Young integration and the convergence of  $\int_0^\cdot \langle {:} \tilde B^{\rho_N} (u_s, u_s) {:}, f  \rangle_{V^{\ast}, V} \mathd s$ to $I (\langle {:} B (\cdummy, \cdummy) {:}, f \rangle_{V^{\ast}, V})$ in $p$-variation. 
\end{proof}

\begin{theorem}[Uniqueness of energy solutions]\label{thm:uniqueness}
  Let $\nu \ll \mu$ with $\frac{\mathd \nu}{\mathd \mu} \in L^2 (\mu)$ and let
  $(u_t)_{t \geqslant 0}$ be an energy solution with $u_0 \sim \nu$. Then $u$
  is a solution to the martingale problem for $(\mathcal{D}_{\max}, \mathcal{L})$, i.e. for any $F \in \mathcal{D}_{\max}$ it holds that
  \[ F (u_t) - F (u_0) - \int_0^t \mathcal{L} F (u_s) \mathd s, \qquad t
     \geqslant 0, \]
  is a martingale in the filtration generated by $u$. Consequently, the law of
  $u$ is uniquely determined by its initial condition, it is a Markov process
  and it has $\mu$ as an invariant measure.
\end{theorem}
{\color{red}
\begin{proof}
  To see that $u$ is a solution to the martingale problem for
  $(\mathcal{D}_{\max}, \mathcal{L})$, let $F \in \mathcal{D}_{\max}$. By
  Lemma~\ref{corepropertycyl}, there exists $F^{\sharp} \in L^2 (\mu)$ with $F
  =\mathcal{R}_1 F^{\sharp}$, and again by Lemma~\ref{corepropertycyl} the set
  $(1 -\mathcal{L}) \mathcal{C}$ is dense in $\mathcal{H}^{- 1}_0$. Therefore,
  there exists a sequence $(F^N)_N \subset \mathcal{C}$ with $(1 -\mathcal{L})
  F^N \rightarrow F^{\sharp}$ in $\mathcal{H}^{- 1}_0$. By
  Proposition~\ref{constructioncandidatesemigroup} we have $\mathcal{R}_1 \in L (\mathcal{H}^{-
  1}_0, \mathcal{H}^1_0)$, and therefore the following convergence holds in
  $\mathcal{H}^1_0$:
  \[ F^N =\mathcal{R}_1 (1 -\mathcal{L}) F^N \rightarrow \mathcal{R}_1
     F^{\sharp} = F, \]
  where the first equality follows from the uniqueness of solutions to the
  resolvent equation in Lemma~\ref{corepropertycyl}, as
  \[ (1 -\mathcal{L}) F^N = (1 -\mathcal{L}) (\mathcal{R}_1 (1 -\mathcal{L})
     F^N) . \]
  We now use this convergence to justify that $u$ solves the martingale
  problem for the test function $F$. By the It{\^o} formula in
  Lemma~\ref{lem:Ito}, the process
  \[ F^N (u_t) - F^N (u_0) - I (\mathcal{L}F^N)_t, \qquad t \geqslant 0, \]
  is a continuous martingale, and we will pass to the limit $N \rightarrow
  \infty$ in each term. First observe that $F^N (u_t) - F^N (u_0) \rightarrow
  F (u_t) - F (u_0)$ in $L^1 (\mathbb{P})$, since by the uniform $L^2$ bound for the density
  \ref{Incompressibility}.,
  \[ |\mathbb{E}[(F - F^N) (u_t) - (F - F^N) (u_0)] | \lesssim \|F - F^N
     \|_{L^2 (\mu)} \lesssim \|F - F^N \|_{\mathcal{H}^1_0} \rightarrow 0. \]
  For the integral term, we use that $(1 -\mathcal{L}) F = F^{\sharp}$ and
  combine this with the energy estimate \ref{Energy estimate}. to obtain
  \begin{align*}
    \mathbb{E} [|I (\mathcal{L}(F - F^N))_t |] & =\mathbb{E} [|I ((F - F^N) -
    (F^{\sharp} - (1 -\mathcal{L}) F^N))_t |]\\
    & \lesssim \|(F - F^N) - (F^{\sharp} - (1 -\mathcal{L})
    F^N)\|_{\mathcal{H}^{- 1}_0} \rightarrow 0.
  \end{align*}
  The martingale property is preserved under $L^1$-limits and thus $u$ is a
  solution to the martingale problem for $(\mathcal{D}_{\max}, \mathcal{L})$.
  
  Now uniqueness of $u$, Markov property and invariance of $\mu$ follow by the
  same arguments as in \cite[Theorem~4.8]{Gubinelli2020}: The martingale problem
  extends to time-dependent test functions $F \in C^1_T L^2 (\mu) \cap C_T
  \mathcal{D}_{\max}$ by writing $F (t, u_t) - F (0, u_0) = \sum_{k = 0}^{n -
  1} (F (t_{k + 1}^n, u_{t^n_{k + 1}}) - F (t^n_k, u_{t^n_k}))$ for $t^n_k = k
  t / n$, and using a Taylor expansion of $F$ in $t$, and the martingale
  problem for time-independent test functions in the $u$ variable. This shows
  that
  \[ F (t, u_t) - F (0, u_0) - \int_0^t (\partial_s +\mathcal{L}) F (s, u_s)
     \mathd s, \qquad t \in [0, T], \]
  is a continuous martingale, and in particular it has vanishing expectation.
  Now it suffices to take $F (t, u) = P_{T - t} G (u)$ for $G \in
  \mathcal{D}_{\max}$ and $(P_t)_{t \geqslant 0}$ the semigroup from
  Proposition~\ref{constructioncandidatesemigroup}, which leads to
  \[ \mathbb{E} [G (u_T)] =\mathbb{E} [P_T G (u_0)], \]
  so we obtain the uniqueness of one-dimensional marginals. The usual iteration as in
  \cite[Theorem~4.4.2]{Ethier1986} or \cite[Theorem~4.8]{Gubinelli2020} then
  yields uniqueness of the finite-dimensional distributions.
\end{proof}
}

\begin{remark}[Ergodicity of energy solutions]\label{rmk:ergodicity}
  Assume $v = 0$ is the only element of $V$ with $\| v \|_{\dot{V}} = 0$.
  Then the invariant measure $\mu$ is ergodic for the Markov process $u$ from
  the previous theorem. Indeed, it suffices to show that if $F \in
  \mathcal{D}_{\max}$ solves $\mathcal{L}F = 0$, then $F$ is a.s. constant.
  But we know from Lemma~\ref{corepropertycyl} that
  \[ 0 = \langle -\mathcal{L}F, F \rangle_{\mathcal{H}^{- 1}_0,
     \mathcal{H}^1_0} = \langle -\mathcal{L}_0 F, F \rangle_{\mathcal{H}^{-
     1}_0, \mathcal{H}^1_0} \simeq \mathbb{E} [\| D F \|_{\dot{V}}^2], \]
  so by assumption $D F = 0$ a.s. and thus $F$ is a.s. constant.
\end{remark}

%% file: Chapters/5_Examples.tex
\section{Examples}\label{sec:examples}

Here we discuss in detail the examples mentioned in the introduction. In each example we identify $H,V,A,B,(\rho_N)$ and verify the conditions in Assumptions~\ref{ass:A} and~\ref{ass:B}. For Assumption~\ref{ass:B} we do not explicitly verify that  $\| B (v, w) \|_{\dot{V}^{\ast}} \lesssim \|v\|_V \|w\|_H + \|v\|_H \|w\|_V$, and instead we only show the stronger condition $\| B (\varphi) \|_{\dot{V}^{\ast}} \lesssim \| \varphi \|_{V \otimes_s H}$, which controls also $B(v,w) = B (v \otimes w)$.

\begin{example}[Multi-component Burgers equation with Dirichlet boundary
conditions on $(0, 1)$ or $(0, \infty)$]\label{ex:Dirichlet}
  Consider the multi-component stochastic Burgers equation on $D \in \{ (0,
  1), (0, \infty) \}$ with homogeneous Dirichlet boundary conditions: Formally,
  $u : \mathbb{R}_+ \times D \rightarrow \mathbb{R}^d$ with
  \[ \partial_t u^i = \Delta u^i + {: }\sum_{j, k} \Gamma^i_{j k} \partial_x
     (u^j u^k) {:} + \sqrt{2 (- \Delta)} \xi^i, \]
  where $\Gamma$ is symmetric in its $3$ arguments. We choose $H = L^2 (D
  \times \{ 1, \ldots, d \}) = L^2 (D)^d$, the smooth functions are $ S =
  C^{\infty}_c (D \times \{ 1, \ldots, d \}) = C^{\infty}_c (D)^d$, and $V =
  H^1_0 (D \times \{ 1, \ldots, d \}) = H^1_0 (D)^d$ with norm
  \[ \| v \|_V^2 = \sum_j (\| v^j \|_{L^2}^2 + \| \partial_x v^j \|_{L^2}^2),
  \]
  and $\| v \|_{\dot{V}}^2 = \| v \|_V^2 - \| v
  \|_H^2 = \sum_j \| \partial_x v^j \|_{L^2}^2$.

\textcolor{red}{We define approximations similarly to ~\cite{Goncalves2020} by averaging over a rescaled indicator function with ''slowed" shifting near the boundary in order to stay inside of the domain.
More precisely, set $x_{\text{max}}=\sup_{y\in D}y\in\{1,\infty\}$, and with interpretation $\infty+r=\infty$, $r\in\mathbb{R}$, we let $p_N : D \to (0,x_{\text{max}}-1/N)$ be a $C^1$, monotone function such that $p_N(x) = x$ for $x \in (0,\,x_{\text{max}}-2/N)$. We additionally assume $p_N$ satisfies
\begin{equation}
0 < \inf_{x,N} p_N'(x)
\leqslant \sup_{x,N}  p_N'(x) = 1.\label{pN_properties}
\end{equation}
We define
\[ \rho_N h (x) = \int_D N \tmmathbf{1}_{I^N_x} (y) h (y) \mathd y,\qquad \text{for} \qquad I_x^N = 
     (p_N(x), p_N(x) + 1 / N) ,
\]
which can be written as
\begin{equation}\label{eq: rhoN as convolution}
    \rho_N h=(\iota_N\ast h)\circ p_N,\qquad \text{for} \qquad \iota_N=N 1_{(-1,0)}(N\cdot).
\end{equation}
}
We will also consider the Sobolev norms
\[ \| v \|_{H^{\alpha}}^2 = \left\{\begin{array}{ll}
     \| v \|_{L^2}^2, & \alpha = 0,\\
     \| v \|_{L^2}^2 + \int_{D^2} \frac{| v (x) - v (y) |^2}{| x - y |^{1 +
     \alpha}} \mathd x \mathd y, & \alpha \in (0, 1),\\
     \| v \|_{L^2}^2 + \| \partial_x v \|_{L^2}^2, & \alpha = 1,
   \end{array}\right. \]
extended to vector-valued functions $v$ as above, and $H^\alpha_0$ is the closure of the smooth compactly supported functions in $H^\alpha$. The homogeneous norms $\| v \|_{\dot H^{\alpha}}^2$ are defined by omitting the squared $L^2$ norm on the right hand side.

Let us verify the assumptions on $A = - \Delta$: For $f, g \in  S$ we obtain,
  using that $f^j, g^j$ are compactly supported in $D$ to justify the
  integration by parts,
  \[ \langle A f, g \rangle_{V^{\ast}, V} = \sum_j \langle \partial_x f^j,
     \partial_x g^j \rangle_{L^2} = \langle f, A g \rangle_{V^{\ast}, V}, \]
  and thus
  \[
    \langle A f, g \rangle_{V^{\ast}, V}  =  \sum_j \langle \partial_x f^j,
    \partial_x g^j \rangle_{L^2}  \leqslant  \left( \sum_j \| \partial_x f^j \|^2_{L^2} \right)^{1 / 2}
    \left( \sum_j \| \partial_x f^j \|^2_{L^2} \right)^{1 / 2} \simeq  \| f \|_{\dot{V}} \| g \|_{\dot{V}},
  \]
  and similarly
  \[ \langle A f, f \rangle_{V^{\ast}, V} = \sum_j \| \partial_x f^j
     \|^2_{L^2} = \| f \|_{\dot{V}}^2 . \]
  To verify the assumptions on $B (f, g) = \sum_{j, k} \Gamma^i_{j k}
  \partial_x (f^j g^k)$, we first observe that $V \otimes_s H \subset H
  \otimes_s H \simeq L^2_s ((D \times \{ 1, \ldots, d \})^2)$, where the
  subscript $s$ means that we consider functions that are symmetric in the
  sense that $\varphi^{i_1, i_2} (x_1, x_2) = \varphi^{i_2, i_1} (x_2, x_1)$,
  and
  \[ \| \varphi \|_{V \otimes_s H}^2 = \sum_{i_1, i_2} \int_{D^2} \left( |
     \varphi^{i_1, i_2} (x, y) |^2 + \frac{1}{2} (| \partial_x \varphi^{i_1,
     i_2} (x, y) |^2 + | \partial_y \varphi^{i_1, i_2} (x, y) |^2) \right)
     \mathd x \mathd y \simeq \| \varphi \|_{H^1_0 (D^2)^{d^2}}^2, \]
  so that we can identify $V \otimes_s H$ with a symmetric subspace of $H^1_0
  (D^2)^{d^2}$. With the trace operator $T\varphi(x) = \varphi(x,x)$ we obtain for smooth and compactly supported $f, \varphi$ and $\alpha \in [0,1]$ 
    \begin{align}\label{estimateB}
    \langle B (\varphi), f \rangle_{V^{\ast}, V} & = \sum_{i, j, k}
    \Gamma^i_{j k} \int_D \partial_x \varphi^{j, k} (x, x) f^i (x) \mathd x \lesssim \sum_{i, j, k} \|\partial_x T \varphi^{j,k}\|_{(\dot{H}^\alpha_0)^\ast} \|f_i\|_{\dot{H}^\alpha}.
  \end{align}
  By integration by parts we see that the derivative $\partial_x$ is in $L(\dot{H}^1,L^2)\cap L(L^2,(\dot{H}^1_0)^\ast)$, so by interpolation it is also in $L(H^{1-\alpha},(\dot{H}^{\alpha}_0)^\ast)$. Moreover, by Lemma \ref{lem:trace} the trace operator is continuous from $H^{\beta} (D^2)$ to $H^{\beta - \frac{1}{2}} (D)$. Therefore, we can bound for $\alpha<1$
  \begin{equation}\label{esttrace}
    \|\partial_x T \varphi\|_{(\dot{H}^\alpha_0)^\ast} \lesssim \| T\varphi\|_{H^{1-\alpha}} \lesssim \| \varphi\|_{H^{3/2 - \alpha}(D^2)}.
  \end{equation}
  For $\alpha=1/2$ we obtain the required estimate for $B:V\otimes_s H \to \dot{V}^\ast$. Moreover, using the symmetry of $\Gamma$ in its arguments $i, j, k$ and
  integration by parts, we see that for $f \in  S$
  \[ \langle B (f, f), f \rangle_{V^{\ast}, V} = \sum_{i, j, k} \Gamma^i_{j k}
     \int_D \partial_x (f^j (x) f^k (x)) f^i (x) \mathd x = 0, \qquad f \in  S.
  \]
  Our remaining task is to verify that the approximation $(\rho_N)$ has the required properties. We know from Lemma~\ref{lem:rhoN-bounds} that $\sup_N \| \rho_N \|_{L (H, H)} + \| \rho_N \|_{L (\dot{V}, \dot{V})} < \infty$ and $\|\rho_N\|_{L(H,V)} \lesssim N$, and for $f \in  S$ and $\varphi \in \mathcal C^\infty_c(D^2)$ we obtain
  \begin{align*}
    \| \rho_N f - f \|_V & \lesssim N^{-1} \|f\|_{H^2},
  \end{align*}
  {\color{red} as well as by (\ref{estimateB}) with $\alpha=0$, Lemma~\ref{lem:rhoN-bounds} and (\ref{esttrace}),
  \begin{align*}
    \langle B(\varphi),(\rho_N-\text{Id})f\rangle_{V^\ast,V} \lesssim \|\varphi\|_{H^{3/2}(D^2)}\|(\rho_N-\text{Id})f\|_{L^2}\lesssim N^{-1}\|\varphi\|_{H^{3/2}(D^2)}\|f\|_{H^1} ,
  \end{align*}}
  and also for any $\kappa \in (0,\tfrac12)$,
  \begin{align}
    | \langle B(\rho_N^{\otimes 2} \varphi-\varphi), f \rangle_{V^{\ast}, V} | & \lesssim  \|\rho_N^{\otimes 2} \varphi-\varphi\|_{H^{1/2+\kappa}(D^2)} \|f\|_{\dot{H}^{1-\kappa}} \lesssim  N^{\kappa-1/2}\|\varphi\|_{H^1(D^2)}\|f\|_{H^{1}}, \label{eq:Dirichlet-B-bound}
  \end{align}
  where the last step follows by another application of Lemma~\ref{lem:rhoN-bounds} together with a tensorization argument.
  Therefore, Assumptions~\ref{ass:A} and~\ref{ass:B} are satisfied in this example. The conditions for our existence theorem are satisfied if $D=(0,1)$ (with $e_k(x) = \sqrt{2} \sin(k\pi x)$), but not for $D=(0,\infty)$.
\end{example}
{\color{red}
\begin{remark}
  Although uniqueness holds on non-compact domains such as $(0,\infty)$
  (since it only relies on Assumptions~(A) and~(B)), the existence result of
  Theorem~\ref{thm:exisence} does not apply in that setting because here we assume
  that $A$ admits an orthonormal basis of eigenvectors, which fails
  for the Laplacian on $(0,\infty)$. See Remark~\ref{rmk:discretize-A} for a possible strategy that could prove the existence of energy solutions also in that case.
\end{remark}
}

To simplify the presentation we restrict to scalar-valued equations from now on, but in all of the following examples it would also be possible to consider systems of equations. 

\begin{example}[Burgers equation with elliptic differential operator and
Dirichlet Boundary conditions]
  We consider the same equation as before, again on a domain $D \in \{ (0, 1),
  (0, \infty) \}$, and for simplicity we assume $d = 1$ (scalar-valued
  equation). We change the operator to
  \[ A v = - \partial_x (a \partial_x v), \]
  for a measurable function $a$ such that there exists $\varepsilon > 0$ with
  $a (x) \in [\varepsilon, \varepsilon^{- 1}]$ for all $x \in D$. The spaces
  $H,  S, V, \dot{V}$ have the same definition as before, as does the approximation $(\rho_N)$. Since $f, g \in  S$
  have compact support, we obtain
  \[\langle A f, g \rangle_{V^\ast, V} = - \int_D \partial_x (a \partial_x f) (x) g (x) \mathd x = \int_D a(x) \partial_x f (x) \partial_x g (x) \mathd x,\]
  and thus
  \[ \langle A f, g \rangle_{V^{\ast}, V} \leqslant \varepsilon^{-1} \|
     \partial_x f \|_{L^2} \| \partial_x g \|_{L^2} \lesssim \| f \|_{\dot{V}}
     \| g \|_{\dot{V}}, \qquad \langle A f, f \rangle_{V^{\ast}, V} \geqslant \varepsilon \| \partial_x
     f \|^2_{L^2} = \varepsilon \| f \|_{\dot{V}}^2 . \]
\end{example}

\begin{example}[Hyperviscous Burgers equation with Neumann boundary conditions on {$(0,1)$}]
  Let $\theta \in (1, 2]$. We consider the equation (for simplicity
  scalar-valued)
  \[ \partial_t u = - (1 - \Delta)^{\theta} u + {:} \left( \partial_x u^2 -
     \frac{2}{3} u^2 (\delta_1 - \delta_0) \right) {:} + \sqrt{2 (1 -
     \Delta)^{\theta}} \xi \]
  on $(0, 1)$, where $\Delta$ is the Laplacian with homogeneous Neumann
  boundary conditions and we define $(1 - \Delta)^{\theta}$ using functional
  calculus, which can be made explicit with the eigenbasis $(\cos (\pi k
  \cdummy))_{k \in \mathbb{N}_0}$ of $\Delta$. The term $- \frac{2}{3} u^2
  (\delta_1 - \delta_0)$ represents a boundary renormalization that is
  necessary to guarantee the property $\langle B (f, f), f \rangle_{V^{\ast},
  V} = 0$, and for $\theta = 1$ (which we exclude!) we could formally interpret it as a change of
  homogeneous Neumann boundary conditions to nonlinear Robin type boundary
  conditions $u' (1) = \frac{2}{3} u^2 (1)$ and $u' (0) = \frac{2}{3} u^2
  (0)$, see~\cite{Goncalves2020} for a detailed discussion
  in a related context. We take $H = L^2 ((0, 1))$, the smooth functions are
  $ S = \{ f \in C^{\infty} ([0, 1]) : f' (0) = f' (1) = 0 \}$, and $V$ is the closure of $ S$ in $H^{\theta} ((0, 1))$, where with the orthonormal basis of eigenfunctions $(e_k)_{k \in \N_0}$ of $(1-\Delta)$ given by $e_k = \sqrt 2 \cos(k\pi x)$ for $k>0$ and $e_0 = 1$,
  \[
    \| f\|_{H^\alpha}^2 := 2 \sum_{k=0}^\infty (1+k^2)^{\alpha} \langle f, e_k\rangle_H^2,
  \]
  and $H^\alpha$ consists of those functions in $L^2$ for which the norm is finite. Due to the factor $2$ in front of the sum and since $(e_k)$ is an orthonormal basis in $H$, we have $\dot V = V$.   

  We use the spectral approximation
  \[
    \rho_N h(x) = \sum_{k=0}^N e_k(x) \langle h, e_k \rangle_H.
  \]
  The estimates for $A$ follow by a simple explicit computation using the orthonormal basis $(e_k)$, so let us focus on the bound for $B$: We can identify $V \otimes_s H$ with a subspace of $H^{\theta} ((0, 1)^2)$, and for smooth
  $\varphi$ in this space and for $f \in  S$ we have with $T \varphi (x) =
  \varphi (x, x)$ for any $\alpha \in [0,1)$ and any $\kappa>0$
  \begin{align*}
    \langle B (\varphi), f \rangle_{V^{\ast}, V} & =  \int_0^1 \partial_x \varphi (x, x) 
    f (x) \mathd x - \frac{2}{3} (\varphi (1, 1) f (1) - \varphi (0, 0) f
    (0))\\
    & \lesssim  \| \partial_x T \varphi \|_{H^{-\alpha}} \| f
    \|_{H^{\alpha}} + \| \varphi \|_{\infty} \| f \|_{\infty} \lesssim  \| T \varphi \|_{H^{3/2+\kappa-\alpha}} \| f \|_{H^{\alpha}} +
    \| \varphi \|_{\infty} \| f \|_{\infty} \\
    & \lesssim  \| \varphi \|_{H^{2+\kappa-\alpha}} \| f \|_{H^{\alpha}} + \| \varphi \|_{\infty} \| f \|_{\infty}, 
  \end{align*}
  where we applied Lemma~\ref{lem:Neumann-derivative} from the Appendix to bound $\partial_x$ (here we need $\alpha<1+\kappa$) and we used that $T: H^{2+\kappa-\alpha} \to H^{3/2+\kappa-\alpha}$ as $3/2+\kappa-\alpha > 0$ (see the Fourier based computation in Lemma~\ref{lem:trace} which works similarly for the Neumann Sobolev spaces). Moreover, for any $\beta>1$ we can bound $\|\varphi\|_\infty \lesssim \|\varphi\|_{H^\beta}$ and $\|f\|_\infty \lesssim \|f\|_{H^{\beta/2}}$, again by a spectral computation. Therefore, using that $\theta>1$, we can find a small $\varepsilon>0$ such that
  \[
    \langle B (\varphi), f \rangle_{V^{\ast}, V}\lesssim \| \varphi \|_{H^{\theta-\varepsilon}} \| f \|_{H^{\theta-\varepsilon}}.
  \]
  Moreover,  due to the boundary renormalization we obtain for all $f \in  S$
  \[ \langle B (f, f), f \rangle_{V^{\ast}, V} = \int_0^1 \frac{2}{3}
     \partial_x f^3 (x) \mathd x - \frac{2}{3} (f (1)^3 - f (0)^3) = 0. \]
  Another simple Fourier computation shows that $\rho_N$ satisfies the same bounds as in Lemma~\ref{lem:rhoN-bounds}, actually for all $0\leqslant \alpha \leqslant \beta$, without requiring $\beta\leqslant 1$. Therefore, the bounds $\|\rho_N f - f\|_{V} \lesssim N^{-\theta} \|f\|_{H^{2\theta}}$ and
  \[
    |\langle B (\rho_N^{\otimes 2}\varphi), f \rangle_{V^{\ast}, V} - \langle B (\varphi), f \rangle_{V^{\ast}, V}| \lesssim N^{-\varepsilon} \|\varphi\|_{V\otimes_s H} \|f\|_{\dot V}
  \]
  follow as in Example~\ref{ex:Dirichlet}. 
     
     We have shown that Assumptions~\ref{ass:A} and~\ref{ass:B} are satisfied in this example, and we already heavily used that $A$ can be diagonalized, so that we also obtain the existence of energy solutions. Unfortunately we required $\theta > 1$ both for estimating the integral part and the boundary part of $B$, and thus we are unable to solve the non-fractional stochastic Burgers
  equation with (renormalized) homogeneous Neumann boundary conditions. 
\end{example}

\begin{example}[Stochastic Burgers equation with regional fractional
Laplacian]
  We consider the domain $D = (0, 1)$ and $\gamma \in \left[ \frac{3}{2}, 2
  \right)$ and the equation studied in \cite{Cardoso2024},
  formally $u : \mathbb{R}_+ \times D \rightarrow \mathbb{R}$,
  \[ \partial_t u =\mathbb{L}^{\gamma / 2} u + \partial_x u^2 + \sqrt{2
     (-\mathbb{L}^{\gamma / 2})} \xi, \]
  where $\mathbb{L}^{\gamma / 2}$ is the regional fractional Laplacian, given for some fixed $c > 0$ by
  \[ \mathbb{L}^{\gamma / 2} f (x) \assign \lim_{\varepsilon \rightarrow 0} c
     \int_0^1 \tmmathbf{1}_{\{ | y - x | \geqslant \varepsilon \}} \frac{f (y)
     - f (x)}{| y - x |^{1 + \gamma}} \mathd y, \]
  which is well-defined and continuous as a uniform limit if
  \[ f \in  S \assign \{ f \in C^{\infty} ([0, 1]) : f (0) = f' (0) = f (1) =
     f' (1) = 0 \} . \]
  Let $H = L^2 (D)$ and $V$ is the closure of $S$ in the
  Sobolev-Slobodeckij/Besov space $H^{\gamma / 2} (D) = B^{\gamma / 2}_{2, 2}
  (D)$ that already appeared in Example~\ref{ex:Dirichlet}. 
  As $\gamma / 2 < 1$, the constraint $f' (0) = f' (1) = 0$ on the derivative
  of $f \in  S$ has no effect on the closure, which equals $H^{\gamma / 2}_0$.  We have that $\| f \|_{\dot{V}}^2 = \int_0^1 \int_0^1
  \frac{(f (y) - f (x))^2}{| y - x |^{1 + \gamma}} \mathd x \mathd y$. We use the same approximation $\rho_N$ as in Example~\ref{ex:Dirichlet}.
  
  For $f,
  g \in  S$ we have the following integration by parts rule from \cite[Thm.~3.3]{Guan2006}:
  \[ \int_0^1 (-\mathbb{L}^{\gamma / 2}) f (x) g (x) \mathd x = \frac{c}{2}
     \int_0^1 \int_0^1 \frac{(f (y) - f (x)) (g (y) - g (x))}{| y - x |^{1 +
     \gamma}} \mathd x \mathd y \lesssim \| f \|_{\dot{V}} \| g \|_{\dot{V}} .
  \]
  For $f = g$, we obtain
  \[ \int_0^1 (-\mathbb{L}^{\gamma / 2}) f (x) f (x) \mathd x \simeq \| f
     \|_{\dot{V}}^2 . \]
  The estimates for $B$ are similar as in the case of the (usual) Dirichlet
  Laplacian on $(0, 1)$: We can again identify $V \otimes_s H$ with a subspace
  of $H^{\gamma / 2}_0 (D^2)$, and from~\eqref{eq:Dirichlet-B-bound} we obtain for any $\kappa \in (0,1/2)$ the bound on the approximation
  \begin{align}\label{eq:Regional-B-rho-bound}
    &| \langle B(\rho_N^{\otimes 2}\varphi), f \rangle_{V^{\ast}, V} -  \langle B(\varphi), f \rangle_{V^{\ast}, V} |  \leqslant \|\rho_N^{\otimes 2} \varphi-\varphi\|_{H^{1/2+\kappa}(D^2)} \|f\|_{\dot{H}^{1-\kappa}},
  \end{align}
  and similarly without approximation
  \begin{align}\label{fracdireasy}
    | \langle B(\varphi), f \rangle_{V^{\ast}, V} | \leqslant \|\varphi\|_{H^{1/2+\kappa}(D^2)} \|f\|_{\dot{H}^{1-\kappa}}.
  \end{align}
  This second estimate leads to the requirement $\tfrac\gamma2 \leqslant \tfrac34$ or $\gamma \leqslant \tfrac32$,  because if $\gamma$ is smaller than that we cannot achieve $\tfrac12 + \kappa \leqslant \tfrac\gamma2$ and $1-\kappa\leqslant \tfrac\gamma2$ simultaneously. Note that $\gamma=\tfrac32$ corresponds to the scaling critical case. To get a small factor from~\eqref{eq:Regional-B-rho-bound} we take $\kappa \in (0,\tfrac14)$ if $f \in  S$, and we take $\kappa \in (\tfrac14,\tfrac12)$ if $\varphi \in H^{\gamma}_0(D^2) \subset  S \otimes_s  S$, and we apply Lemma~\ref{lem:rhoN-bounds}. {\color{red} The small factor for $\langle B(\varphi), (\rho_N-\text{Id})\cdot \rangle_{V^{\ast}, V}$ follows similarly from (\ref{fracdireasy}) and Lemma~\ref{lem:rhoN-bounds}.}
  Finally, we obtain for $f \in  S$
  \[ \langle B (f, f), f \rangle_{V^{\ast}, V} = \int_D \partial_x (f (x) f
     (x)) f (x) \mathd x = \frac{2}{3} \int_D \partial_x f^3 (x) \mathd x = 0,
     \qquad v \in V. \]
  Therefore, Assumptions~\ref{ass:A} and~\ref{ass:B} are satisfied. Also, $A$ is diagonalizable because $V$ is compactly embedded in $H$, and therefore our existence proof applies. But actually the existence of energy solutions for this example was previously shown by Cardoso and Gon\c{c}alves~\cite{Cardoso2024}, who derive them as scaling limits of fluctuations in long range interacting particle systems. They consider only the stationary case, but the derivation should extend without problems to initial conditions with $L^2$ density or even $L^1$ density as in~\cite{Graefner2024}.

  Apart from $A$ being diagonalizable, everything should extend also to $(0,\infty)$ instead of $(0,1)$.
\end{example}

%% file: Chapters/6_Auxiliary_Computations.tex
\appendix\section{Auxiliary computations}

\begin{lemma}
  \label{lem:B-circular}The assumption $\langle B (v, v), v \rangle_{V^{\ast},
  V} \equiv 0$ together with the symmetry and bilinearity of $B$ implies
  \[ \langle B (v_1, v_2), v_3 \rangle_{V^{\ast}, V} + \langle B (v_3, v_1),
     v_2 \rangle_{V^{\ast}, V} + \langle B (v_2, v_3), v_1 \rangle_{V^{\ast},
     V} = 0 \]
  for all $v_1, v_2, v_3 \in V$.
\end{lemma}

\begin{proof}
  By assumption we have for $v, w \in V$
  \begin{align*}
    0 & =  \langle B (v + w, v + w), v + w \rangle_{V^{\ast}, V} + \langle B
    (v - w, v - w), v - w \rangle_{V^{\ast}, V}\\
    & =  \langle B (v, v), w \rangle_{V^{\ast}, V} + 2 \langle B (v, w), v
    \rangle_{V^{\ast}, V} + 2 \langle B (v, w), w \rangle_{V^{\ast}, V} +
    \langle B (w, w), v \rangle_{V^{\ast}, V}\\
    &  \quad - \langle B (v, v), w \rangle_{V^{\ast}, V} - 2 \langle B (v, w), v
    \rangle_{V^{\ast}, V} + 2 \langle B (v, w), w \rangle_{V^{\ast}, V} +
    \langle B (w, w), v \rangle_{V^{\ast}, V}\\
    & =  4 \langle B (v, w), w \rangle_{V^{\ast}, V} + 2 \langle B (w, w), v
    \rangle_{V^{\ast}, V},
  \end{align*}
  so that the claim follows if $v_1 = v_2$, i.e.
  \[ \langle B (w, w), v \rangle_{V^{\ast}, V} = - 2 \langle B (v, w), w
     \rangle_{V^{\ast}, V} . \]
  Now we apply this identity repeatedly to obtain
  \begin{align*}
    &  2 \langle B (v_1, v_2), v_3 \rangle_{V^{\ast}, V} + \langle B (v_1,
    v_1), v_3 \rangle_{V^{\ast}, V} + \langle B (v_2, v_2), v_3
    \rangle_{V^{\ast}, V}\\
    & = \langle B (v_1 + v_2, v_1 + v_2), v_3 \rangle_{V^{\ast}, V}\\
    &  = - 2 \langle B (v_1 + v_2, v_3), v_1 + v_2 \rangle_{V^{\ast}, V}\\
    &  = - 2 \left( \langle B (v_1, v_3), v_1 \rangle_{V^{\ast}, V} +
    \langle B (v_2, v_3), v_2 \rangle_{V^{\ast}, V} + \langle B (v_1, v_3),
    v_2 \rangle_{V^{\ast}, V} + \langle B (v_2, v_3), v_1 \rangle_{V^{\ast},
    V} \right)\\
    &  = \langle B (v_1, v_1), v_3 \rangle_{V^{\ast}, V} + \langle B (v_2,
    v_2), v_3 \rangle_{V^{\ast}, V} - 2 (\langle B (v_1, v_3), v_2
    \rangle_{V^{\ast}, V} + \langle B (v_2, v_3), v_1 \rangle_{V^{\ast}, V}),
  \end{align*}
  and subtracting $\langle B (v_1, v_1), v_3 \rangle_{V^{\ast}, V} + \langle B
  (v_2, v_2), v_3 \rangle_{V^{\ast}, V}$ on both sides and dividing by $2$, we
  end up with
  \[ \langle B (v_1, v_2), v_3 \rangle_{V^{\ast}, V} = - \langle B (v_1, v_3),
     v_2 \rangle_{V^{\ast}, V} - \langle B (v_2, v_3), v_1 \rangle_{V^{\ast},
     V} . \qedhere\]
\end{proof}

We would like to give an alternative proof of
Lemma~\ref{lem:second-Malliavin}, without relying on the chaos expansion. This
is based on the following lemma:

\begin{lemma}
  \label{Ncalculus}{\tmdummy}
  \[ \mathcal{N}D_h = D_h (\mathcal{N}- 1) \tmmathbf{1}_{\mathcal{N} \geqslant
     1}, \qquad \mathcal{N} \delta = \delta (\mathcal{N}+ 1). \]
\end{lemma}

\begin{proof}
  By the commutation relation of $D_h$ and $\delta$, see in
  \cite[eq.~(1.46)]{Nualart2006}, we have
  \begin{align*}
    {}[\mathcal{N}, D_h] F & =  \delta D D_h F - D_h \delta D F\\
    & =  \delta D_h D F - D_h \delta D F\\
    & =  D_h \delta D F - \langle D F, h \rangle_H - D_h \delta D F\\
    & =  - D_h F.
  \end{align*}
  The second claim follows by duality of $\delta$ and $D$.
\end{proof}

\begin{proof}[Alternative proof of Lemma~\ref{lem:second-Malliavin}]
  Applying the previous result, we have
  \begin{align*}
    \mathbb{E} [\| D^2 F \|_{H \otimes_s \dot{V}}^2] & = \sum_{\ell_1,
    \ell_2} \mathbb{E} \left[ \left\langle D D_{e_{\ell_1}} F, D
    D_{e_{\ell_2}} F \right\rangle_H \right] \langle e_{\ell_1}, e_{\ell_2}
    \rangle_{\dot{V}}\\
    & = \sum_{\ell_1, \ell_2} \mathbb{E} \left[ D_{e_{\ell_1}}
    F\mathcal{N}D_{e_{\ell_2}} F \right] \langle e_{\ell_1}, e_{\ell_2}
    \rangle_{\dot{V}}\\
    & = \sum_{\ell_1, \ell_2} \mathbb{E} \left[ \left( \mathcal{N}^{1 / 2}
    D_{e_{\ell_1}} F \right) \mathcal{N}^{1 / 2} D_{e_{\ell_2}} F \right]
    \langle e_{\ell_1}, e_{\ell_2} \rangle_{\dot{V}}\\
    & = \sum_{\ell_1, \ell_2} \mathbb{E} \left[ D_{e_{\ell_1}}
    ((\mathcal{N}- 1)^{1 / 2} \tmmathbf{1}_{\mathcal{N} \geqslant 1} F)
    D_{e_{\ell_2}} ((\mathcal{N}- 1)^{1 / 2} \tmmathbf{1}_{\mathcal{N}
    \geqslant 1} F) \right] \langle e_{\ell_1}, e_{\ell_2} \rangle_{\dot{V}}\\
    & = \mathbb{E} \left[ \left\| D (\mathcal{N}- 1)^{1 / 2}
    \tmmathbf{1}_{\mathcal{N} \geqslant 1} F \right\|_{\dot{V}}^2  \right] .\qedhere
  \end{align*}
\end{proof}

\begin{lemma}
  \label{lem:tensor-operator}\textcolor{red}{Let $\rho\in L
     (V, V)\cap L(H, H)$.} The tensorization $\rho^{\otimes 2}$ of $\rho$
  satisfies the following bound:
  \[ \| \rho^{\otimes 2} \varphi \|_{V \otimes_s H} \leqslant \| \rho \|_{L
     (V, V)} \| \rho \|_{L (H, H)} \| \varphi \|_{V \otimes_s H}, \qquad
     \varphi \in V \otimes_s H. \]
\end{lemma}

\begin{proof}
  \begin{align*}
    \| \rho^{\otimes 2} \varphi \|_{V \otimes_s H}^2 & =  \sum_k \|
    (\rho^{\otimes 2} \varphi) (\cdummy, e_k) \|_V^2\\
    & \leqslant  \| \rho \|_{L (V, V)}^2 \sum_k \| ((\tmop{id} \otimes \rho)
    \varphi) (\cdummy, e_k) \|_V^2\\
    & \leqslant  \| \rho \|_{L (V, V)}^2 \| \rho \|_{L (H, H)}^2 \sum_k \|
    \varphi (\cdummy, e_k) \|_V^2\\
    & =  \| \varphi \|_{V \otimes_s H}^2 . \qedhere
  \end{align*}
\end{proof}

{\color{red}
\begin{lemma}\label{lem:core}
    If $S$ is a core for $A$, then $\mathcal C$ is a core for $\mathcal L_0$.
\end{lemma}

\begin{proof}
  Let $F \in \mathcal{D} (\mathcal{L}_0)$ and let $(T_t)_{t \geqslant 0}$ be
  the semigroup generated by $\mathcal{L}_0$. The construction in Section~2 of
  \cite{Shigekawa1992} gives for $G (\eta) = \Psi (\eta (g)) \in
  \mathcal{C}$ the Mehler formula
  \begin{equation}\label{eq:Mehler}
    T_t G (\eta) = \int G \left( e^{- t A} \eta + \sqrt{1 - e^{- 2 t A}}
    \zeta \right) \mu (\mathd \zeta) \assign \int \Psi \left( \eta (e^{- t A} g) + \zeta \left( \sqrt{1 -
    e^{- 2 t A}} g \right) \right) \mu (\mathd \zeta),
  \end{equation}
  where $(e^{- t A})_{t \geqslant 0}$ is the semigroup generated by $A$. In
  the next steps, we build approximations in $\mathcal{C}$ that converge to
  $F$ in graph norm.
  \begin{enumerate}
    \item Approximation of $F$ by $A_t F$: Consider $A_t F \assign \frac{1}{t}
    \int_0^t T_s F \mathd s$, which converges to $F$ in graph norm as $t
    \rightarrow 0$: Indeed, $A_t F \rightarrow F$ in $L^2 (\mu)$ by strong
    continuity of $(T_t)_{t \ge 0}$  and
    \[ \| \mathcal{L}_0 A_t F -\mathcal{L}_0 F \| = \left\| \frac{1}{t} (T_t F
       - F) -\mathcal{L}_0 F \right\| \rightarrow 0. \]
    
    \item Approximation of $A_t F$ by $A_t F^N \in A_t \mathcal{C}$: Since
    $\mathcal{C}$ is dense in $\mathcal{H}^1_0$ by definition, we can find
    $(F^N)_N \subset \mathcal{C}$ with $\| F^N - F \|_{\mathcal{H}^1_0}
    \rightarrow 0$. We claim that $(A_t F^N)$ converges to $A_t F$ in graph
    norm as $N\to\infty$, for all $t > 0$. Indeed, $\| A_t F^N - A_t F \|_{\mathcal{H}^0_0}
    \rightarrow 0$ for all $t > 0$ because $A_t$ is a contraction, and also
    \[ \| \mathcal{L}_0 A_t F^N -\mathcal{L}_0 A_t F \| = \left\| \frac{1}{t}
       (T_t F^N - F^N) - \frac{1}{t} (T_t F - F) \right\| \rightarrow 0. \]
    
    \item Approximation of $A_t F^N$ by $A^{(m)}_t F^N \in \mathcal{C} (D
    (A))$: For $t \geqslant 0$ and $G \in \mathcal{C}$ of the form $G(\eta)=\Psi(\eta(g))$ the Mehler formula shows that $T_t G(\eta) = \Phi(\eta(e^{- t A} g))$ with the polynomial
    \[
        \Phi(x) = \int \Psi \left( x + \zeta \left( \sqrt{1 -
    e^{- 2 t A}} g \right) \right) \mu (\mathd \zeta).
    \]
    
    Moreover, $e^{- t A} S \subset e^{- t A} D (A) \subset D (A)$, where $D
    (A)$ is the domain of $A$. Therefore,
    \[ T_t \mathcal{C} \subset \mathcal{C} (\mathcal{D} (A)) \assign \left\{ G
       \in L^2 (\Omega) : G (\eta) = \Phi (\eta (f)), f \in \mathcal{D} (A)^m,
       \Phi : \mathbb{R}^m \rightarrow \mathbb{R} \text{ polynomial}, m \in
       \mathbb{N} \right\} . \]
    Now consider the Riemann sum approximation of $A_t$:
    \[ A^{(m)}_t \assign \frac{1}{m} \sum_{k = 0}^{m - 1} T_{\frac{k t}{m}} .
    \]
    Then $A^{(m)}_t F^N \in \mathcal{C} (D (A))$ for all $m, N \in
    \mathbb{N}$, and we claim that $A^{(m)}_t F^N \rightarrow A_t F^N$ in
    graph norm:
    \[ \|A^{(m)}_t F^N - A_t F^N \| \leqslant \frac{1}{t} \sum_{k = 0}^{m - 1}
       \int_{\frac{k t}{m}}^{\frac{(k + 1) t}{m}} \left\| \left( T_{\frac{k
       t}{m}} - T_s \right) F^N \right\| \mathd s \rightarrow 0, \qquad m
       \rightarrow \infty, \]
    and since $F^N \in \mathcal{D} (\mathcal{L}_0)$ also
    \[ \|\mathcal{L}_0 A^{(m)}_t F^N -\mathcal{L}_0 A_t F^N \| = \|A^{(m)}_t
       \mathcal{L}_0 F^N - A_t \mathcal{L}_0 F^N \| \rightarrow 0, \qquad m
       \rightarrow \infty . \]
    \item Approximation of $A^{(m)}_t F^N$ by $\mathcal{C}$: Since $A^{(m)}_t
    F^N \in \mathcal{C} (\mathcal{D} (A))$, there exist $k \in \mathbb{N}$, a
    polynomial $\Psi : \mathbb{R}^k \rightarrow \mathbb{R}$, and $g \in D
    (A)^k$ such that $A^{(m)}_t F^N (\eta) = \Psi (\eta (g))$. Since $S$ is a
    core for $A$, there exist $(g^i)_{i \in \mathbb{N}} \subset S^k$ with $g^i
    \rightarrow g$ and $A g^i \rightarrow A g$ in $H$. Then $(\Psi (\eta
    (g^i)))_i \in \mathcal{C}$ converges in graph norm to $\Psi (\eta (g)) =
    A^{(m)}_t F^N (\eta)$ as $i\to \infty$: The converge in $L^2 (\mu)$ is clear, and
    \begin{align*}
      \mathcal{L}_0 \Psi (\eta (g^i)) & = \sum_{\ell = 1}^k \partial_{\ell}
      \Psi (\eta (g^i)) (\eta (- A g^i_{\ell})) + \sum_{\ell_1, \ell_2 = 1}^k
      \partial_{\ell_1, \ell_2} \Psi (\eta (g^i)) \langle - A g^i_{\ell_1},
      g^i_{\ell_2} \rangle\\
      & \rightarrow \sum_{\ell = 1}^k \partial_{\ell} \Psi (\eta (g)) (\eta
      (- A g_{\ell})) + \sum_{\ell_1, \ell_2 = 1}^k \partial_{\ell_1, \ell_2}
      \Psi (\eta (g)) \langle - A g_{\ell_1}, g_{\ell_2} \rangle
      =\mathcal{L}_0 A^{(m)}_t F^N (\eta).
    \end{align*}
  \end{enumerate}
  The combination of Steps 1.-4. shows that $F$ can be approximated by elements
  of $\mathcal{C}$ in graph norm, so $\mathcal{C}$ is a core for
  $\mathcal{D} (\mathcal{L}_0)$.
\end{proof}
}

\begin{lemma}[Bounds for $\rho_N$ of Example~\ref{ex:Dirichlet}]\label{lem:rhoN-bounds}
  Let $0 \leqslant \alpha \leqslant \beta \leqslant 1$ and let $D$ and $\rho_N$ be as in Example~\ref{ex:Dirichlet}. Then
  \begin{align*}
    \| \rho_N v \|_{H^{\beta}} & \lesssim N^{\beta - \alpha} \| v
    \|_{H^{\alpha}}, \\
    \| \rho_N v - v \|_{H^{\alpha}} & \lesssim N^{- (\beta - \alpha)} \| v
    \|_{H^{\beta}},
  \end{align*}
  and the same bounds also hold if we replace $H^\alpha$ and $H^\beta$ by their homogeneous counterparts $\dot{H}^\alpha$ and $\dot{H}^\beta$.
\end{lemma}

\begin{remark}
    Before going to the proof, we stress that the specific form of $I^N_x$ is never used in the arguments, and we only need that $y \in I^N_x$ forces $|\textcolor{red}{p_N(x)}-y| \leqslant 1/N$ and that $|I^N_x \cap D| = 1$ for all $x\in D$.
\end{remark}

\begin{proof}
  We use interpolation arguments to reduce to the special cases $\alpha, \beta
  \in \{ 0, 1 \}$. We focus on the inhomogeneous estimates, the homogeneous estimate follows from the same arguments, since along the way we show $\|\rho_N v\|_{\dot{H}^1} \lesssim \|v\|_{\dot{H}^1}$ and $\|\rho_N v - v\|_{L^2} \lesssim \|v\|_{\dot{H}^1}$.
  \begin{enumeratenumeric}
    \item Estimates for $\| \rho_N v \|_{H^{\beta}}^2$: For $\alpha = \beta =
    0$ we use that $N \tmmathbf{1}_{I^N_x}$ defines a probability measure for
    each $x \in D$ and we have $y \in I^N_x$ only if $|\textcolor{red}{p_N(x)} - y | <
    \frac{1}{N}$, and thus by Jensen's inequality and Fubini's theorem
    \textcolor{red}{
    \begin{align}
        \| \rho_N v \|_{L^2}^2 & = \int_D \left( \int_D N \tmmathbf{1}_{(p_N(x),p_N(x)+1/N)}(y) v (y) \mathd y \right)^2 \mathd x\nonumber \\
        & \leqslant \int_D \int_D N \tmmathbf{1}_{(p_N(x),p_N(x)+1/N)}(y) |v (y)|^2 \mathd y \mathd x\nonumber \\
        & \leqslant N\int_D \int_{p_N(D)}\tmmathbf{1}_{(x,x+1/N)}(y)\frac{1}{p_N'(x)}       \mathd x| v (y) |^2\mathd y\nonumber\\
        &\lesssim \| v\|_{L^2}^2,
    \end{align}
    where we performed a change of variables in the third step and then used \eqref{pN_properties}.
    }
 \textcolor{red}{Using an implicit smooth approximation of $v$ and $\iota_N$,    
    the convolution identity \eqref{eq: rhoN as convolution}, a change of variables and Young's convolution inequality gives
    \begin{align}\label{eq: a=0 b=1}
        \|\partial_x \rho_Nv\|_{L^2}^2&=\int_D|(v\ast\partial_x\iota_{N})(p_N(x))p_N'(x)|^2\,\mathd x\nonumber\\
        &\lesssim N^2 \int_{D}|v(p_N(x)+1/N)-v(p_N(x))|^2\,\mathd x\nonumber\\
        &\lesssim N^2 \|v\|_{L^2}^2\,\mathd x,
    \end{align}
    and so
    \begin{equation}\label{eq: bound for H1 norm of rhoN v by l2 norm of v}
         \| \partial_x \rho_N v \|_{L^2}^2 \lesssim N^2 \| v \|_{L^2}^2 \qquad
       \Rightarrow \qquad \| \rho_N v \|_{H^1} \lesssim N \| v \|_{L^2} .
     \end{equation} 
}  
    \textcolor{red}{By repeating the same computation as in \eqref{eq: a=0 b=1} but putting the derivative on $v$, it follows that 
     \[ \| \rho_N v \|_{H^1} \lesssim \| v \|_{H^1} . \]
    }
    Let $X_0 = X_1 = L^2$ and $Y_0 = L^2$, $Y_1 = H^1$, so that $\| \rho_N
    \|_{L (X_0, Y_0)} \lesssim 1$ and $\| \rho_N \|_{L (X_1, Y_1)} \lesssim
    N$. For $\theta = \beta$ we have $[X_0, X_1]_{\theta} = L^2$ and $[Y_0,
    Y_1]_{\theta} = H^{\beta}$, so the interpolation theorem
    yields for some $C > 0$ that does not depend on
    $N$
    \begin{equation}
      \| \rho_N \|_{L (L^2, H^{\beta})} \leqslant C^{1 - \theta} (C
      N)^{\theta} \simeq N^{\beta} . \label{eq:interpolation-L2-Hb}
    \end{equation}
    Interpolation with $X_0 = Y_0 = L^2$ and $X_1 = Y_1 = H^1$ gives $[X_0,
    X_1]_{\theta} = [Y_0, Y_1]_{\theta} = H^{\beta}$, and thus
    \begin{equation}
      \| \rho_N \|_{L (H^{\beta}, H^{\beta})} \lesssim 1.
      \label{eq:interpolation-Hb-Hb}
    \end{equation}
    Now we interpolate~\eqref{eq:interpolation-L2-Hb}
    and~\eqref{eq:interpolation-Hb-Hb} with $\theta = \frac{\beta -
    \alpha}{\beta}$ and $X_0 = L^2$, $Y_0 = H^{\beta}$ and $X_1 = Y_1 =
    H^{\beta}$, so that $[X_0, X_1]_{\theta} = H^{\beta (1 - \theta)} =
    H^{\alpha}$ and $[Y_0, Y_1]_{\theta} = H^{\beta}$ and
    \[ \| \rho_N \|_{L (H^{\beta}, H^{\beta})} \lesssim N^{\beta \theta} =
       N^{\beta - \alpha} . \]
    \item Estimates for $\| \rho_N v - v \|_{H^{\alpha}}^2$: For $\alpha =
    \beta = 0$ and $\alpha = \beta = 1$ we simply use the triangle inequality.
    For $\alpha = 0$ and $\beta = 1$ we bound, writing $(x, y) = (y, x)$ if $y
    < x$,

    \textcolor{red}{
    \begin{align*}
      \| \rho_N v - v \|_{L^2}^2 & =  \int_D \left( \int_D N
      \tmmathbf{1}_{I^N_x} (y) (v (y) - v (x)) \mathd y \right)^2 \mathd x\\
      & \leqslant \int_D \int_D N \tmmathbf{1}_{I^N_x} (y) \left( \int_x^y
      \partial_z v (z) \mathd z \right)^2 \mathd y \mathd x\\
      & \leqslant  \int_D \int_D \int_D N | y - x |\tmmathbf{1}_{I^N_x} (y)
      \tmmathbf{1}_{(x, y)} (z) | \partial_z v (z) |^2 \mathd z \mathd y
      \mathd x\\
      & \leqslant  \int_D \int_D \int_D N | y - x | \tmmathbf{1}_{| p_N(x) - y |
      \leqslant \frac{1}{N}} \tmmathbf{1}_{| z - p_N(x) | \leqslant \frac{2}{N}} |
      \partial_z v (z) |^2 \mathd z \mathd y \mathd x\\
      & \lesssim  N^{- 2} \| \partial_z v \|_{L^2}^2,
    \end{align*}
    where in the penultimate and last step we used that $p_N(x)\leqslant x \leqslant p_N(x)+2/N$, which implies that $|p_N(x)-x|\leqslant 2/N$.
    }
    
    Thus, $\| \rho_N v - v \|_{L^2} \lesssim N^{- 1} \| v \|_{H^1}$.
    Interpolation with $X_0 = Y_0 = L^2$ and $X_1 = H^1$, $Y_1 = L^2$ and $1 -
    \theta = \beta$ yields
    \begin{equation}
      \| \rho_N - \tmop{id} \|_{L (H^{\beta}, L^2)} \lesssim N^{- \beta} .
      \label{eq:interpolation-diff-Hb-L2}
    \end{equation}
    Taking instead $Y_1 = H^1$, we get
    \begin{equation}
      \| \rho_N - \tmop{id} \|_{L (H^{\beta}, H^{\beta})} \lesssim 1.
      \label{eq:interpolation-diff-Hb-Hb}
    \end{equation}
    Interpolating~\eqref{eq:interpolation-diff-Hb-L2}
    and~\eqref{eq:interpolation-diff-Hb-Hb} with $\theta =
    \frac{\alpha}{\beta}$, we finally obtain
    \[ \| \rho_N - \tmop{id} \|_{L (H^{\beta}, H^{\alpha})} \lesssim N^{-
       \beta (1 - \theta)} = N^{- (\beta - \alpha)} . \qedhere \]
  \end{enumeratenumeric}
\end{proof}

\begin{lemma}[Trace operator on Sobolev spaces]
  \label{lem:trace}Let $D \subset \mathbb{R}^d$ be a Lipschitz domain and let
  $s > \frac{d}{2}$. Then the trace
  \[ T \varphi (x) \assign \varphi (x, x), \qquad x \in D, \]
  defines a bounded linear operator $T : H^s (D \times D) \rightarrow H^{s - d
  / 2} (D)$.
\end{lemma}

\begin{proof}
  According to \cite[Ch.~4.1]{Triebel2006} we have
  \[ \| T \varphi \|_{H^{s - d / 2} (D)} \simeq \inf \{ \| \psi \|_{H^{s - d /
     2} (\mathbb{R}^d)} : \psi |_D = T \varphi \}, \]
  and similarly for $\| \varphi \|_{H^s (D \times D)}$. If $\psi \in H^s
  (\mathbb{R}^{2 d})$ is such that $\psi |_{D \times D} = \varphi$, then $T
  \psi |_D = T \varphi$ and thus
  \[ \| T \varphi \|_{H^{s - d / 2} (D)} \lesssim \inf \{ \| T \psi \|_{H^{s -
     d / 2} (\mathbb{R}^d)} : \psi |_{D \times D} = \varphi \} . \]
  But on $\mathbb{R}^d$ we can access Fourier methods to compute $\mathcal{F}
  (T \psi) (z) = \int_{\mathbb{R}^d} \mathcal{F} \psi (z - z', z') \mathd z'$
  and then
  \begin{align*}
    & \| T \psi \|_{H^{s - d / 2} (\mathbb{R}^d)}^2\\
    & = \int_{\mathbb{R}^d} (1 + | z |^2)^{s - \frac{d}{2}} | \mathcal{F} (T
    \psi) (z) |^2 \mathd z\\
    & \leqslant \int_{\mathbb{R}^d} (1 + | z |^2)^{s - \frac{d}{2}} \left(
    \int_{\mathbb{R}^d} (1 + | z - z' |^2 + | z' |^2)^{- s} \mathd z' \right)
    \left( \int_{\mathbb{R}^d} (1 + | z - z' |^2 + | z' |^2)^s | \mathcal{F}
    \psi (z - z', z') |^2 \mathd z' \right) \mathd z\\
    & \lesssim \int_{\mathbb{R}^d} (1 + | z |^2)^{s - \frac{d}{2}} (1 + | z
    |^2)^{\frac{d}{2} - s} \int_{\mathbb{R}^d} (1 + | z - z' |^2 + | z' |^2)^s
    | \mathcal{F} \psi (z - z', z') |^2 \mathd z' \mathd z\\
    & = \| \psi \|_{H^s (\mathbb{R}^d \times \mathbb{R}^d)},
  \end{align*}
  using $s > \frac{d}{2}$ in the fourth line to bound the integral
  $\int_{\mathbb{R}^d} (1 + | z - z' |^2 + | z' |^2)^{- s} \mathd z'$. Thus,
  we have shown that
  \[ \| T \varphi \|_{H^{s - d / 2} (D)} \lesssim \inf\{ \| \psi \|_{H^s
     (\mathbb{R}^d \times \mathbb{R}^d)} : \psi |_{D \times D} = \varphi \} =
     \| \varphi \|_{H^s (D \times D)} . \qedhere\]
\end{proof}

\begin{lemma}[Regularity of the derivative on Neumann Sobolev spaces]
  \label{lem:Neumann-derivative}Consider for $\theta \in \mathbb{R}$ the
  Neumann Sobolev spaces
  \[ H^{\theta} = \left\{ h = \sum_{k = 0}^{\infty} \cos (\pi k \cdummy)
     \hat{h} (k) : \| f \|_{H^{\theta}}^2 = \sum_{k = 0}^{\infty} (1 + | k
     |^2)^{\theta} | \hat{h} (k) |^2 < \infty \right\}, \]
  and for $\theta < 0$ we interpret $h \in H^{\theta}$ as an element of the
  dual space of $H^{- \theta}$. Let $\theta \in \left( \frac{1}{2},
  \frac{5}{2} \right)$ and $\delta > \frac{3}{2}$. Then the space derivative
  $\partial_x$ is a bounded linear operator from $H^{\theta}$ to $H^{\theta -
  \delta}$.
\end{lemma}

\begin{proof}
  We consider $h$ of the form $h = \sum_{k = 0}^N \cos (\pi k \cdummy) \hat{h}
  (k)$ and derive an estimate that is uniform in $N$. The Fourier type
  coefficients of $\partial_x h$ are
  \begin{align*}
    | \widehat{\partial_x h} (k) | & =  \left| \int_0^1 \cos (\pi k x)
    \sum_{\ell = 0}^N (- \pi \ell) \sin (\pi \ell x) \hat{h} (\ell) \mathd x
    \right|\\
    & =  \left| \sum_{\ell = 0}^N (- \pi \ell) \hat{h} (\ell) \int_0^1 \cos
    (\pi k x) \sin (\pi \ell x) \mathd x \right|\\
    & =  \left| \sum_{\ell = 0}^N (- \pi \ell) \hat{h} (\ell)
    \tmmathbf{1}_{\ell \neq k} \frac{\ell (1 - (- 1)^{k + \ell})}{\pi (\ell^2
    - k^2)} \right|\\
    & \lesssim  \left( \sum_{\ell = 0}^N (1 + | \ell |^2)^{\theta} | \hat{h}
    (\ell) |^2 \right)^{1 / 2} \left( \sum_{\ell > 0, \ell \neq k}
    \frac{\ell^{4 - 2 \theta}}{(\ell - k)^2 (\ell + k)^2} \right)^{1 / 2}\\
    & \lesssim  \| h \|_{H^{\theta}} \left( \sum_{\ell < k} \frac{\ell^{4 -
    2 \theta}}{(\ell - k)^2 k^2} + \sum_{\ell > k} \frac{\ell^{2 - 2
    \theta}}{(\ell - k)^2} \right)^{1 / 2}\\
    & \lesssim  \| h \|_{H^{\theta}} ((1 + | k |^2)^{1 - \theta})^{1 / 2},
  \end{align*}
  where the last step follows by comparing the sums with integrals and using
  that $\theta \in \left( \frac{1}{2}, \frac{5}{2} \right)$. Therefore,
  \begin{align*}
    \| \partial_x h \|_{H^{\theta - \delta}}^2 & =  \sum_{k = 0}^{\infty} (1
    + | k |^2)^{\theta - \delta} | \widehat{\partial_x h} (k) |^2\\
    & \lesssim  \| h \|_{H^{\theta}}^2 \sum_{k = 0}^{\infty} (1 + | k
    |^2)^{\theta - \delta} (1 + | k |^2)^{1 - \theta}\\
    & =  \| h \|_{H^{\theta}}^2 \sum_{k = 0}^{\infty} (1 + | k |^2)^{1 -
    \delta},
  \end{align*}
  and the series on the right hand side is finite if $\delta > \frac{3}{2}$.
\end{proof}

%% file: Chapters/7_Acknowledgments.tex
\paragraph{Acknowledgements}
The authors are grateful to Xiaohao Ji, Pedro Cardoso and Patr\'{i}cia Gon{\c{c}}alves for fruitful discussions. A special thanks goes to Zhilin Yang for pointing out errors related to the assumptions on and construction of $\rho_N$.\\
The first two authors are grateful for funding by DFG through EXC 2046, Berlin Mathematical School, and through IRTG 2544 “Stochastic Analysis in Interaction”. The first author expresses his graditude for funding from the UKRI Future Leaders Fellowship MR/W008246/1, ``Large-scale universal behaviour of
Random Interfaces and Stochastic Operators".
The second author also acknowledges support from DFG CRC/TRR 388 ``Rough Analysis, Stochastic Dynamics and Related Fields", Project A01.
The final author was supported by the EPSRC
Centre for Doctoral Training in Mathematics of Random Systems: Analysis, Modelling and Simulation
(EP/S023925/1), and now acknowledges funding from the SDAIM project ANR-22-CE40-0015 funded by
the French National Research Agency (ANR).